\documentclass[smallextended]{svjour3}
\usepackage[utf8]{inputenc}
\usepackage[english]{babel}
\usepackage[margin=1in]{geometry}
\usepackage[T3,T1]{fontenc}
\usepackage{mathptmx}
\usepackage{graphicx}
\usepackage{psfrag}
\usepackage{color}
\usepackage{amsmath}
\usepackage{amscd}
\usepackage{amssymb}
\usepackage[section]{placeins}
\usepackage[normalem]{ulem}
\usepackage{caption}
\usepackage{ragged2e}
\usepackage{subfigure}
\usepackage{nicefrac}
\usepackage{bbold}
\usepackage{cancel}
\usepackage{framed}
\usepackage{dsfont}
\usepackage{mathtools}
\usepackage{soul}
\usepackage{authblk}
\usepackage{hyperref}
\usepackage{float}

\DeclareSymbolFont{tipa}{T3}{cmr}{m}{n}
\DeclareMathAccent{\invbreve}{\mathalpha}{tipa}{16}

\DeclareCaptionJustification{reallyjustified}{\justifying}
\spnewtheorem*{theorem*}{Theorem}{\bf}{\it}
\spnewtheorem{TheoremIntro}{Theorem}{\bf}{\it}

\newcommand{\var}{\mathrm{Var}}
\newcommand{\EX}{\mathds{E}}
\newcommand{\expect}[1]{\EX\left(#1\right)}

\newcommand{\prob}[1]{\mathrm{Pr} \left(#1\right)}

\newcommand{\vari}[1]{\var\left(#1\right)}

\newcommand{\clo}[1]{\mathrm{cl}\left( #1 \right)}
\newcommand{\inter}[1]{\mathrm{int}\left( #1 \right)}
\newcommand{\bound}[1]{\mathrm{bd}\left( #1 \right)}
\newcommand{\hull}[1]{\mathcal{H}\left( #1 \right)}

\newcommand{\MS}{M(\mathcal{S})}
\newcommand{\MSF}{M\left(\mathcal{S},\phi\right)}
\newcommand{\MSM}[1]{M_{#1}\left(\mathcal{S},\phi\right)}

\newcommand{\GSF}{\gamma(\mathcal{S},\phi)}

\makeatletter
\newcommand*\bigcdot{\mathpalette\bigcdot@{.5}}
\newcommand*\bigcdot@[2]{\mathbin{\vcenter{\hbox{\scalebox{#2}{$\m@th#1\bullet$}}}}}
\makeatother
\newcommand{\innp}[2]{#1\bigcdot #2}
\DeclareMathOperator{\argmax}{argmax}
\DeclareMathOperator{\argmin}{argmin}

\begin{document}

\title{A convex analysis approach to tight expectation inequalities}
\author{Andr\'e M. Timpanaro}
\institute{A. M. Timpanaro \at Universidade Federal do ABC, 09210-580 Santo Andr\'e, Brazil \\\email{a.timpanaro@ufabc.edu.br}}
\maketitle

\begin{abstract}

In this work, we investigate the question of how knowledge about expectations $\expect{f_i(X)}$ of a random vector $X$ translate into inequalities for $\expect{g(X)}$ for given functions $f_i$, $g$ and a random vector $X$ whose support is contained in some set $\mathcal{S}\subseteq \mathds{R}^n$. We show that there is a connection between the problem of obtaining tight expectation inequalities in this context and properties of convex hulls, allowing us to rewrite it as an optimization problem. The results of these optimization problems not only arrive at sharp bounds for $\expect{g(X)}$ but in some cases also yield discrete probability measures where equality holds.

We develop an analytical approach that is particularly suited for studying the Jensen gap problem when the known information are the average and variance, as well as a numerical approach for the general case, that reduces the problem to a convex optimization; which in a sense extends known results about the moment problem.

\end{abstract}

\section{Introduction}
\label{sec:intro}

\subsection{Setting and Motivation}

Inequalities that relate different expectations that can be computed from a random vector $X$ play a central role in probability theory and its many applications. As examples of these expectation inequalities we can cite the inequalities due to Markov, Chebyshev and Jensen \cite{Markov-1884,Chebyshev-1867,Jensen-1906}. These inequalities can be thought as answering the question of how knowledge about expectations $\expect{f_i(X)}$ of a random vector $X$ translate into bounds for $\expect{g(X)}$, for specific functions $f_i$ and $g$. For example, Jensen's inequality for a convex function $f$ reads

\begin{equation}
\expect{f(X)} \geq f\left(\expect{X}\right)
\label{eq:jensen}
\end{equation}
This allows us to translate knowledge of $\expect{X}$ into a lower bound for $\expect{f(X)}$, in situations where $f$ is convex. In these inequalities, one can also require that the support of $X$ be contained in some set $\mathcal{S}$. A simple example is Markov's inequality:

\begin{equation}
\prob{X>a} \leq \frac{\expect{X}}{a}
\label{eq:markov}
\end{equation}
for a nonnegative random variable $X$ and $a>0$. Drawing a parallel with inequality (\ref{eq:jensen}), inequality (\ref{eq:markov}) translates knowledge of $\expect{X}$ and that the support of $X$ is restricted to $\mathcal{S} = \mathds{R}_+$ into an upper bound for the expectation $\expect{\Theta(x-a)}$, where $\Theta$ is the Heaviside step function:

\begin{equation}
\Theta(x) =\left\{
\begin{array}{ll}
1,  &   \mbox{if }x\geq 0 \\
0,  &   \mbox{if }x<0
\end{array}
\right.
\label{eq:heaviside}
\end{equation}

Looking at these inequalities from this perspective, an interesting question that arises is whether an expectation inequality is the best possible (in the sense that we can get arbitrarily close to having an equality) given the required information. For example, if all we know is that $f$ is a convex function and that $\expect{X} = \phi$, then the lower bound for $\expect{f(X)}$ given by Jensen's inequality (\ref{eq:jensen}) is the best possible because we can always find a probability measure such that equality holds (the measure where $X=\phi$ with probability 1). On the other hand, if $X$ is a nonnegative random variable with $\expect{X} = \phi$, then the upper bound given by Markov's inequality (\ref{eq:markov}) for $\expect{\Theta(x-a)} = \prob{X>a}$ is clearly not the best possible when $0 < a < \phi$, as we must have $\prob{X>a} \leq 1$ (to obtain a bound where equality can always hold we would need to take the minimum between 1 and the bound in (\ref{eq:markov})).

In this work, we will consider the following questions:
\begin{question}
Let $X$ be a random vector with probability measure $\mu$ on $\mathds{R}^n$ (with the Borel $\sigma$-algebra) that has its support contained in a given set $\mathcal{S}\subseteq \mathds{R}^n$. Let $f_1, \ldots , f_m$ and $g$ be measurable functions $\mathds{R}^n\rightarrow \mathds{R}$. If $\expect{f_i(X)}_{\mu} =\phi_i\in\mathds{R}\,,\,\,i = 1,\ldots ,m$, then what lower or upper bounds can we establish for $\expect{g(X)}_{\mu}$?
\label{question1}
\end{question}

\begin{question}
Are the bounds we obtain answering question \ref{question1} the best possible, in the sense that there exists a sequence of measures (satisfying all constraints imposed in question \ref{question1}) $(\mu_k)_{k=0}^{\infty}$ such that

\[
\lim_{k\rightarrow\infty} \expect{g(X)}_{\mu_k} = B
\]
where $B$ is the obtained bound (either a lower or an upper bound).
\label{question2}
\end{question}

Many specific instances of question \ref{question1} have been considered in the literature regarding the problem of bounding the \emph{Jensen gap} \cite{Liao-Berg-2018,Walker-2014,Gao-Sitharam-Roitberg-2019,Abramovich-Persson-2016,Dragomir-2015,Abramovich-Ivelic-Pecaric-2010,Simic-2011,Dragomir-2001,Pecaric-1985,Abramovich-Jameson-Sinnamon-2004}

\[
\expect{f(X)} - f(\expect{X}).
\]
These works can be seen as generalizing the inequality (\ref{eq:jensen}). Since the Jensen bound is the best one given $\expect{X}$ for a convex function, these generalizations require \emph{different input information}. They can be broken down on bounds that require more (or different) information about the function $f$ (like analyticity assumptions \cite{Walker-2014,Abramovich-Persson-2016,Dragomir-2015}, superquadraticity \cite{Abramovich-Ivelic-Pecaric-2010,Abramovich-Jameson-Sinnamon-2004} or assumptions about asymptotic behaviours \cite{Gao-Sitharam-Roitberg-2019}) and bounds that require knowledge of more expectations besides $\expect{X}$ (like the variance \cite{Liao-Berg-2018,Walker-2014,Abramovich-Persson-2016,Dragomir-2015,Dragomir-2001}, other dispersion measures \cite{Gao-Sitharam-Roitberg-2019} or more complicated expectations \cite{Abramovich-Ivelic-Pecaric-2010,Pecaric-1985,Abramovich-Jameson-Sinnamon-2004}). A notable limitation of the current results regarding the Jensen gap is that most of them require the random vector to have a support contained in $\mathds{R}$ or even specific subsets of $\mathds{R}$.

Another famous case where the problem in questions \ref{question1} and \ref{question2} arises is in the theory of Moment Problems. More precisely, the case where the $f_i$ and $g$ are polynomials can be solved numerically (in the sense that the best lower and upper bounds for $\expect{g(X)}$ can be found) using semidefinite programming \cite{moment-book,moment-survey,moment-popescu}.

As such, addressing questions \ref{question1} and \ref{question2} in a more general setting would provide an unifying framework for studying expectation inequalities and their generalizations, as well as possibly giving some new insight into moment problems. It is also worth mentioning that interest in the problem of how certain expectations impact others has increased in some applied fields. As an example we can cite the optimization of thermal machines \cite{Cavina-2016,TUR-de-force}, showing the strong potential for applications of these questions.

\subsection{Main ideas and strategy}

The main insight behind our work is that if we consider the function $\Gamma(x) = (f_1(x), f_2(x), \ldots, f_m(x), g(x))$, then the expectation $\expect{\Gamma(X)}_{\mu}$ must be in the convex hull of $\Gamma(\mathcal{S})$. On the other hand, since we know that $\expect{f_i(X)}_{\mu} = \phi_i$ for $i = 1,\ldots ,m$, then if $\phi = (\phi_1, \ldots, \phi_m)$ it follows that $\expect{\Gamma(X)}_{\mu}$ must also be in the line $\{\phi\} \times \mathds{R}$. Studying the intersection between this line and the convex hull gives us then the possible values of $\expect{\Gamma(X)}_{\mu}$ for measures $\mu$ with support contained in $\mathcal{S}$ that satisfy the constraints on $\expect{f_i(X)}_{\mu}$ (Figure \ref{fig:hull-ineq} illustrates this in the case $m=1$)

\begin{figure}[H]
\centering
\includegraphics[width=0.5\textwidth]{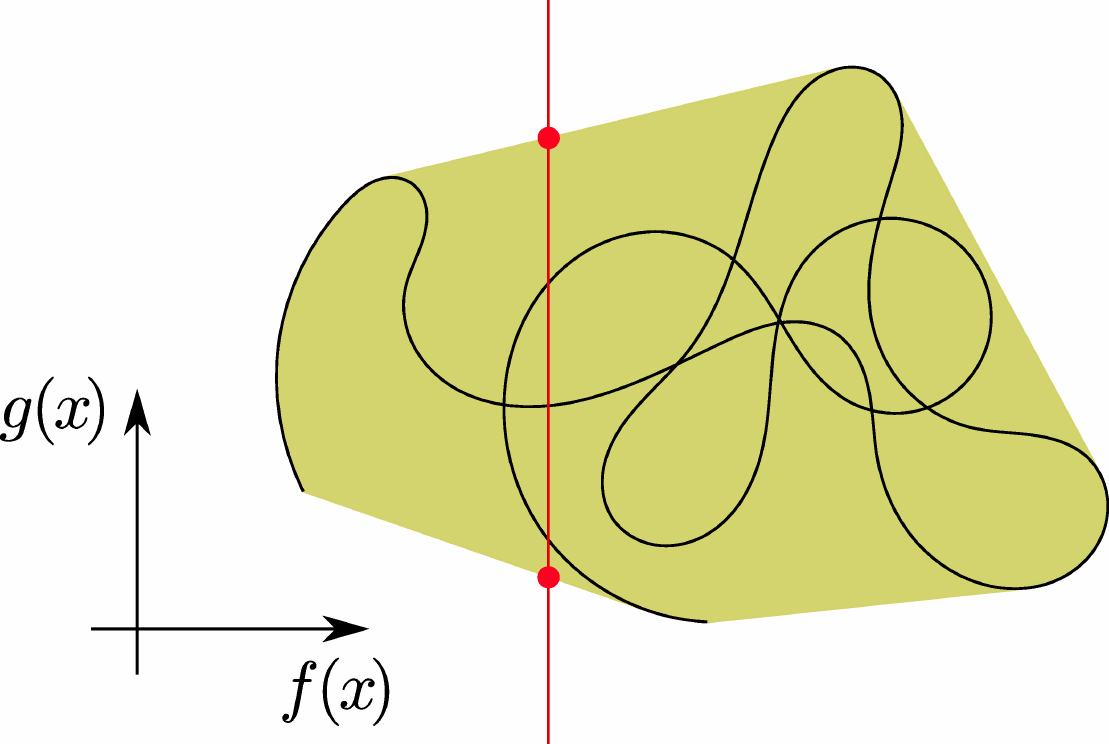}
\caption{Qualitative graph of the curve $\Gamma(x) = (f(x), g(x))$ for $x \in \mathcal{S}$ (in black) and the convex hull of its image $\Gamma(\mathcal{S})$ (in beige). No matter what the measure $\mu$ is, the point $(\expect{f(X)}_{\mu}, \expect{g(X)}_{\mu})$ must lie in the hull. The line $\{\phi\} \times \mathds{R}$ (in red) gives the points obeying the constraint $\expect{f(X)}_{\mu} = \phi$. If the extreme points of the intersection are $(\phi, g_-)$ and $(\phi, g_+)$, this implies the inequality $g_- \leq \expect{g(X)}_{\mu} \leq g_+$ for all measures $\mu$ with $\expect{f(X)}_{\mu} = \phi$ and support in $\mathcal{S}$.}
\label{fig:hull-ineq}
\end{figure}

This observation solves, at least in principle, the question of what are the best lower or upper bounds we can establish for $\expect{g(X)}_{\mu}$. However, the complexity of dealing with convex hulls severely hinders the usefulness of this way of obtaining bounds. As such, most of the work is devoted to theorems that allow these bounds to be obtained in a simpler way, instead of requiring finding the actual convex hull. The main strategies behind these theorems are that we don't need all the information in the convex hull to find the bounds, but only a few of its supporting hyperplanes and that Caratheodory's theorem can be used to bound the cardinality of the supports of the measures we need to consider.

\subsection{Notation, definitions and main results}

Before stating our main results, we'll stablish the notation conventions and definitions that will be used for the rest of this work. The functions $f_i$ will always be the $m$ functions whose expectations are known (bundled together as the vector $f$), $\phi_i$ are the values of these expectations (bundled as the vector $\phi$) and $g$ is the function whose expectation we wish to study. We will always be considering random vectors with probability measures on $\mathds{R}^n$ with the Borel $\sigma$-algebra and whose support is contained in $\mathcal{S}\subseteq \mathds{R}^n$ (although we will also consider on occasion supports contained in sets derived from $\mathcal{S}$). As such $f_i$ and $g$ will always be $\mathds{R}^n\rightarrow \mathds{R}$ measurable functions.
We further define/denote

\begin{itemize}
\item $z$ is the $(m+1)$-th coordinate of a point in $\mathds{R}^{m+1}$.
\item $\Gamma(x) \equiv (f_1(x),f_2(x),\ldots,f_m(x),g(x))$.
\item $\expect{F(X)}_{\mu}$ is the expectation of $F(x)$ with respect to the probability measure $\mu$.
\item $\MS$ is the set of all probability measures $\mu$ on $\mathds{R}^n$, with support contained in $\mathcal{S}\subseteq \mathds{R}^n$.
\item $\MSF$ is the subset of $\MS$ such that $\expect{f(X)}_{\mu} = \phi$ and $\expect{g(X)}_{\mu}$ is finite.
\item $\MSM{k}$ is the subset of $\MSF$ with measures that have at most $k$ points in their support.
\item $\mathcal{H}(S)$ is the convex hull of a set $S$.
\item We will denote by $\mathcal{D}(\mathcal{S})$ the set $\{\phi\,|\,\MSF \neq \varnothing\}$ of all values $\phi$ that give expectation constraints that can actually be satisfied in $\MS$. 
\item $\GSF \equiv \hull{\Gamma(\mathcal{S})} \cap \left(\{\phi\}\times\mathds{R}\right)$.
\item $\clo{S}$ is the closure of a set $S$.
\item $\inter{S}$ is the interior of a set $S$.
\item $\bound{S}$ is the boundary of a set $S$.
\item $\innp{u}{v}$ denotes the scalar product $\sum_i u_i v_i$.
\item $\overline{\mathds{R}}$ denotes the extended real line.
\end{itemize}

\noindent The following definition will also be useful
\begin{definition}[Progressive Cover]
A \ul{progressive cover} of a set $S\subseteq \mathds{R}^n$ is a sequence of sets $\left(S_i\right)_{i=1}^{\infty}$ such that
\begin{itemize}
\item $S_i \subseteq S_j$ if $i\leq j$
\item $S_i \subseteq S\,\,\forall\,i$
\item $\forall\,x\in S\,\exists\,i\,\vert\, x\in S_i$
\end{itemize}

A \ul{progressive compact cover} is a progressive cover where all sets in it are compact and a \ul{progressive bounded cover} is a progressive cover where all sets in it are bounded.
\end{definition}

\subsubsection{Main results}

With the definitions we made, it follows that the possible values for $(\expect{f(X)}_{\mu}, \expect{g(X)}_{\mu})$ where $\mu\in\MSF$ are the points in $\GSF$. If one of the bounds for $\expect{g(X)}_{\mu}$ is finite, then there is an endpoint $(\phi, \sigma)$ of $\GSF$. If $\Pi$ is a supporting hyperplane of $\hull{\Gamma(\mathcal{S})}$, passing through $(\phi, \sigma)$, then the equation defining $\Pi$ can be used to find the bound for $\expect{g(X)}_{\mu}$. Furthermore, if $(\phi, \sigma)$ is in the hull, then there must be a measure $\mu$ with support in $\Pi$ such that $\expect{g(X)}_{\mu} = \sigma$. (Figure \ref{fig:theo-hull} provides an illustration of the construction used in the case $m=1$ for different values of $\phi$)

\begin{figure}[hbtp]
\centering
\includegraphics[width=0.5\textwidth]{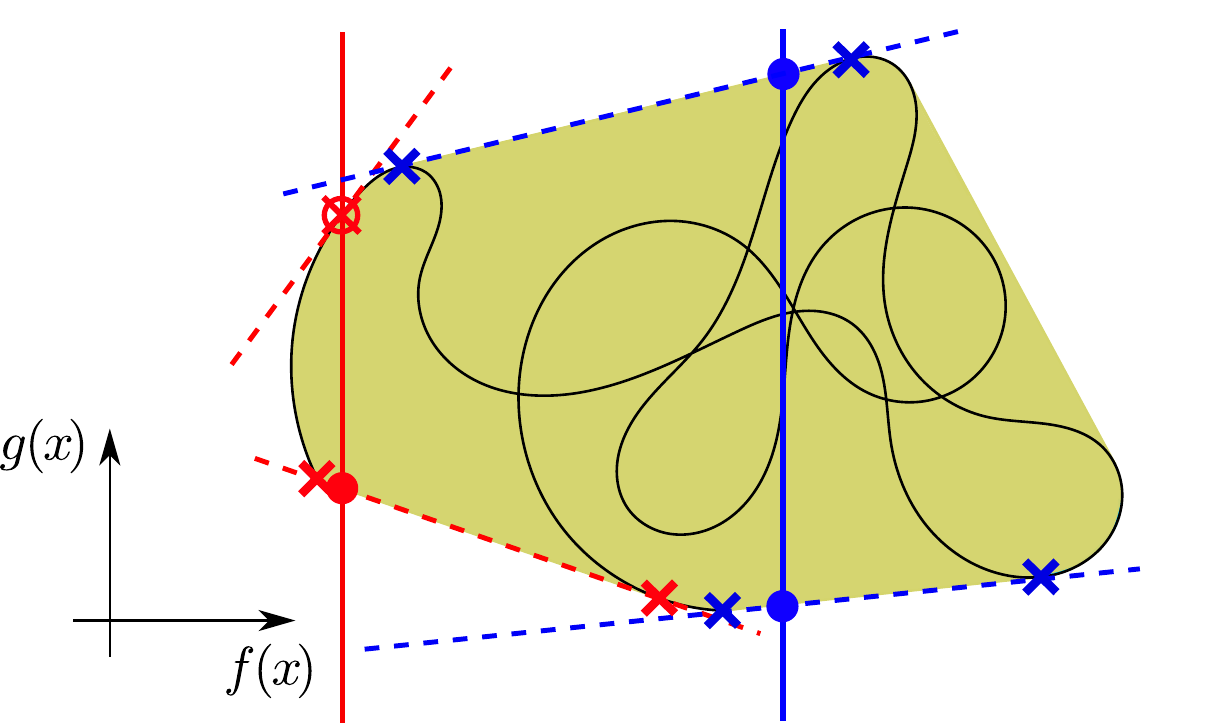}
\caption{Building upon figure \ref{fig:hull-ineq}, the vertical lines are $\{\phi_1\} \times \mathds{R}$ (to the left in red) and $\{\phi_2\} \times \mathds{R}$ (to the right in blue). The circle points are the $(\phi, \sigma)$ points and the dashed lines are the supporting hyperplanes passing through them. Knowing the equation of an hyperplane and the value of $\phi$ allows one to calculate the corresponding bound. In all these cases, the support must be in $\Pi$, so we can actually infer the support of the maximizing/minimizing measure (they correspond to the square points in the figure). So for example, we can maximize $\expect{g(X)}$ given that $\expect{f(X)} = \phi_1$ with a measure whose support is a singleton.
}
\label{fig:theo-hull}
\end{figure}

More precisely, we have the following theorem:

\begin{TheoremIntro}[Proven in section \ref{sec:geom-approach} as theorem \ref{th:analytical}]
If $\mathcal{S} \subseteq \mathds{R}^n$ and
\[
\sigma_{-} = \inf_{\GSF} z\quad\quad\quad\quad \sigma_{+} = \sup_{\GSF} z
\]
then there exist vectors $\alpha_{\pm} = (\alpha_{\pm}^1, \ldots, \alpha_{\pm}^m)$ and values $\beta_{\pm}, c_{\pm}$ such that $\beta_{\pm} \geq 0$, $(\alpha_{\pm}^1, \ldots, \alpha_{\pm}^m, \beta_{\pm})\neq\vec{0}$ and

\begin{itemize}
\item If $\sigma_+$ is finite, then
\[
\innp{\alpha_+}{f(x)} + \beta_+g(x) + c_+ \leq 0\,\forall\, x\in\mathcal{S}\quad\quad\mbox{and}\quad
\sup_{\mu\in\MSF}\expect{\innp{\alpha_+}{f(X)} + \beta_+g(X) + c_+}_{\mu} = 0
\]
\item If $\sigma_-$ is finite, then
\[
\innp{\alpha_-}{f(x)} + \beta_-g(x) + c_- \geq 0\,\forall\, x\in\mathcal{S}\quad\quad\mbox{and}\quad
\inf_{\mu\in\MSF}\expect{\innp{\alpha_-}{f(X)} + \beta_-g(X) + c_-}_{\mu} = 0
\]
\item Moreover, for each case where $\sigma_{\pm}$ is finite, if $(\phi,\sigma_{\pm})\in\GSF$, then there exists a measure $\mu_{\pm}\in\MSM{m+1}$ with support $s_{\pm}$, such that
\[
\expect{g(X)}_{\mu_{\pm}} = \sigma_{\pm}\quad\mbox{and}\quad \innp{\alpha_{\pm}}{f(x)} + \beta_{\pm}g(x) + c_{\pm} = 0\,\forall\, x\in s_{\pm}
\]
\end{itemize}
\label{th-intro:hyperplanes}
\end{TheoremIntro}

In fact, we don't need to use supporting hyperplanes of $\hull{\Gamma(\mathcal{S})}$ at all. As long as $\Pi$ is an hyperplane that separates $\mathds{R}^{m+1}$ in 2 regions, one of which has no intersection with $\Gamma(\mathcal{S})$, and $\Pi$ is not parallel to the $z$ direction, then its equation will provide a bound for $\expect{g(X)}$. We can then rewrite the problem of finding sharp bounds for $\expect{g(X)}$ in terms of an optimization of the parameters defining $\Pi$ (Figure \ref{fig:theo-dual} illustrates this)

\begin{figure}[hbtp]
\centering
\includegraphics[width=0.5\textwidth]{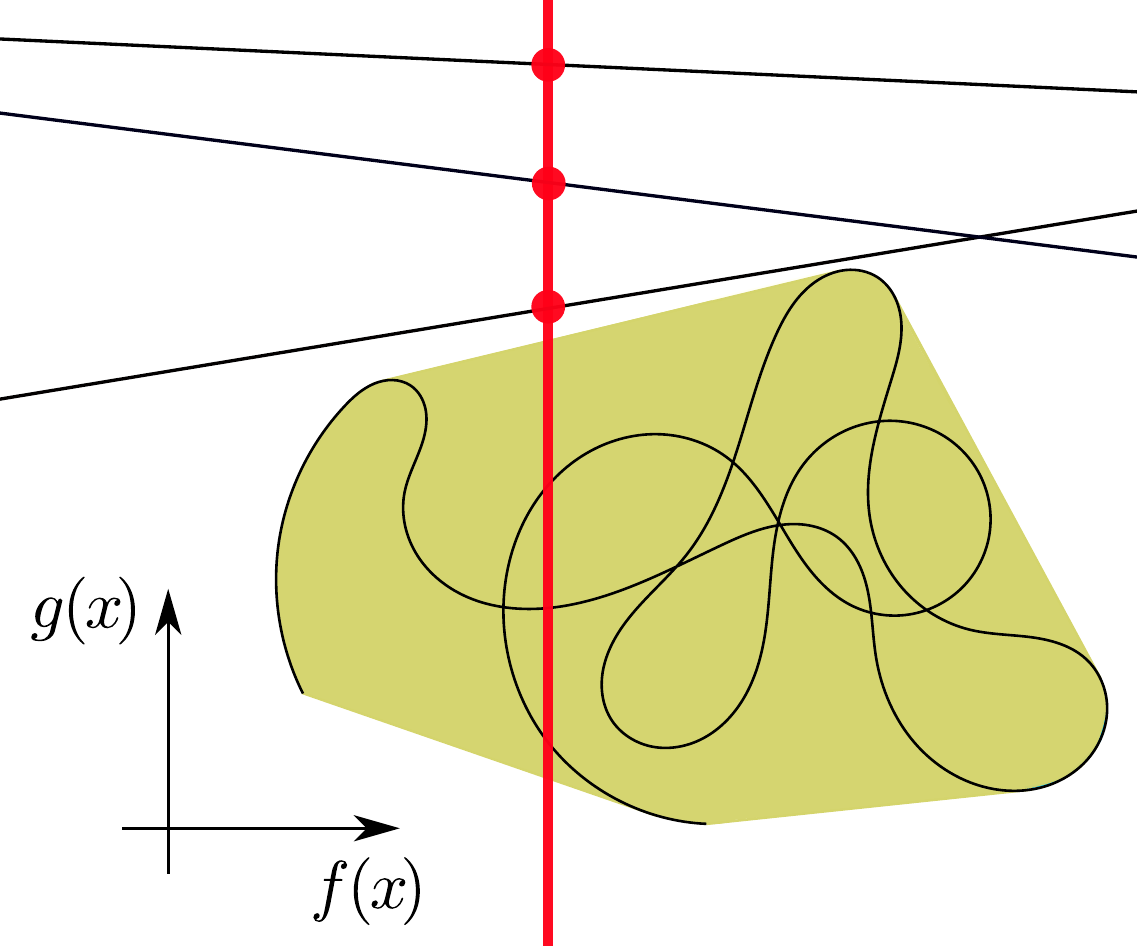}
\caption{Still building upon figure \ref{fig:hull-ineq}, the red line is $\{\phi\} \times \mathds{R}$ and the black lines are candidates for the optimal hyperplane. The red points are of the form $(\phi, \sigma_{\Pi})$, where $\sigma_{\Pi}$ is an upper bound for $\expect{g(X)}$, associated with $\Pi$. Minimizing said upper bound will give us the bound we are actually interested in.
}
\label{fig:theo-dual}
\end{figure}

\begin{TheoremIntro}[Proven in section \ref{sec:numerics} as theorem \ref{th:numerical}]
Let $\mathcal{S}\subseteq\mathds{R}^n$. If $\phi\notin\bound{\mathcal{D}(\mathcal{S})}$, then
\[
\inf_{\mu\in\MSF} \expect{g(X)}_{\mu} = \sup_{\alpha\in\mathds{R}^m}\left(\inf_{x\in\mathcal{S}} \left(g(x) + \innp{\alpha}{(f(x) - \phi)}\right)\right)
\]\[
\mbox{and}\quad\quad\sup_{\mu\in\MSF} \expect{g(X)}_{\mu} = \inf_{\alpha\in\mathds{R}^m}\left(\sup_{x\in\mathcal{S}} \left(g(x) + \innp{\alpha}{(f(x) - \phi)}\right)\right)
\]
\label{th-intro:dual}
\end{TheoremIntro}

Note that these results go beyond the case where $X$ is a single random variable and allow the study of random vectors. Theorem \ref{th:analytical} is more suited for finding analytical bounds, specially in cases with low dimensionality and few expectations known (see section \ref{sssec:applications} for an example). On the other hand, as will be shown later, theorem \ref{th:numerical} rewrites the problem as a convex optimization and as such is more suited for numerics. It will also allow us to access problems with higher dimensionality and larger numbers of expectation constraints. 

\subsubsection{Applications}
\label{sssec:applications}

Using theorem \ref{th:analytical} we can bound the cardinality of the support of our maximizing/minimizing probability measure. If the functions $f_i$ and $g$ have special properties, this cardinality bound can be made more stringent. In the case of the following theorem, using $f_1(x) = x$, $f_2(x)=x^2$, requiring that $g'(x)$ be strictly convex and that the variable be in an interval $[a,b]$, we can show that the minimizing/maximizing measures have at most 2 points in their support, allowing us to recover these measures and calculate the bounds explicitly:

\begin{TheoremIntro}[Proven in section \ref{sec:jensen-gap-conv} as theorem \ref{th:dg-strict-convex}]
Let $X$ be a random variable with support contained in $[a,b]$ and let $g:[a,b] \rightarrow \mathds{R}$ be bounded, differentiable and such that $g'(x)$ is strictly convex. Then for every $\lambda$ and $\sigma^2 > 0$ that are possible values for the average and variance of a variable in $[a,b]$, there exist probability measures $\mu_{\pm}$ with
\[
\expect{X}_{\mu_{\pm}} = \lambda\quad\quad\mbox{and}\quad\quad \vari{X}_{\mu_{\pm}} = \sigma^2
\]
such that
\[
\expect{g(X)}_{\mu_-} = \frac{\sigma^2 g(a) + (\lambda - a)^2 g\left(\lambda + \frac{\sigma^2}{\lambda - a}\right)}{\sigma^2 + (\lambda - a)^2}
\]\[
\expect{g(X)}_{\mu_+} = \frac{\sigma^2 g(b) + (\lambda - b)^2 g\left(\lambda + \frac{\sigma^2}{\lambda - b}\right)}{\sigma^2 + (\lambda - b)^2}
\]
and for every measure $\mu$ in $M([a,b])$, with the same average and variance, we have
\[
\expect{g(X)}_{\mu_-} \leq \expect{g(X)}_{\mu} \leq \expect{g(X)}_{\mu_+}
\]
\label{th-intro:jensen}
\end{TheoremIntro}

This result gives us sharp bounds for $\expect{g(X)}$ in terms of the function $g$, evaluated at points determined entirely by $a$, $b$, $\lambda$ and $\sigma$, which is an useful feature when using it as a bound for the Jensen gap of $g$.

As an application of this theorem we derive bounds for the moment generating function of a positive random variable (section \ref{sssec:mom-gen}). If $X$ is a strictly positive random variable, with average $\lambda$ and variance $\sigma^2$, then we can bound its moment generating function as:

\[
\expect{e^{Xs}} \geq \frac{\sigma^2 + \lambda^2e^{\nicefrac{(\lambda^2 +\sigma^2)s}{\lambda}}}{\sigma^2 + \lambda^2}\quad\quad\mbox{if }s>0\mbox{ and}
\]\[
e^{\lambda s} \leq \expect{e^{Xs}} \leq \frac{\sigma^2 + \lambda^2e^{\nicefrac{(\lambda^2 +\sigma^2)s}{\lambda}}}{\sigma^2 + \lambda^2}\quad\quad\mbox{if }s<0
\]
Similar bounds for the power means $\expect{X^s}^{\nicefrac{1}{s}}$ of a positive variable were also be obtained (section \ref{sssec:pow-mean}).

\subsection{Outline}

In section \ref{sec:hulls-expect} we establish the connection between the problem of bounding expectations and properties of convex hulls, while also making some examples, like a fairly simple generalization of Jensen's inequality corresponding to the optimal bounds when the average is the only information known (section \ref{sec:jensen-example}). In section \ref{sec:geom-approach} we prove theorem \ref{th-intro:hyperplanes} and provide some examples on how to use it, highlighting the importance of progressive covers for applications of this theorem. Specific results for when $\expect{X}$ and $\vari{X}$ are the known information (including the results in section \ref{sssec:applications}) are proven in section \ref{sec:jensen-gap}. Finally, in section \ref{sec:numerics} we prove theorem \ref{th-intro:dual} that reduces the problem of bounding $\expect{g(X)}$ given $\expect{f_i(X)}$ to a convex optimization problem and we provide some examples, including one that involves random vectors.

\section{Convex hulls and expectations}
\label{sec:hulls-expect}

In order to tackle our problem, we first recall the well known lemma

\begin{lemma}

If $\mathcal{S} \subseteq \mathds{R}^n$ and $\mu\in\MS$ has a non-divergent expectation, $\expect{X}_{\mu}$, then $\expect{X}_{\mu}\in\mathcal{H}(\mathcal{S})$.

\label{conv-lemma}
\end{lemma}

The same general reasoning used in proving lemma \ref{conv-lemma} can be used to extend it to situations where we are mapping the random vector:

\begin{lemma}

Let $\mu\in\MS$, with $\mathcal{S} \subseteq \mathds{R}^n$ and let $\Psi: \mathds{R}^n\rightarrow \mathds{R}^m$ be a measurable function such that $\expect{\Psi(X)}_{\mu}$ is non-divergent, 
then $\expect{\Psi(X)}_{\mu} \in \mathcal{H}(\Psi(\mathcal{S}))$.
\label{conv-lemma-psi}
\end{lemma}

From here the following corollary follows:

\begin{corollary}

Let $\Psi: \mathds{R}^n\rightarrow \mathds{R}^m$ be a measurable function, $\mathcal{S} \subseteq \mathds{R}^n$ and $\Xi \in \mathds{R}^m$. Then there exists a measure $\mu \in \MS$ such that

\[
\expect{\Psi(X)}_{\mu} = \Xi
\]
iff $\Xi \in \hull{\Psi(\mathcal{S})}$.

\label{co:existence}
\end{corollary}

\begin{proof}
If $\Xi \in \hull{\Psi(\mathcal{S})}$, then by Caratheodory's theorem, $\Xi$ is a finite convex combination of elements in $\Psi(\mathcal{S})$:

\[
\Xi = \sum_{k} \alpha_k P_k
\]
The measure $\mu\in\MS$ we are after can be obtained by taking $Q_k\in\mathcal{S}$ such that $\Psi(Q_k) = P_k$ and attributing probability $\alpha_k$ to each $Q_k$.

On the other hand, if $\expect{\Psi(X)}_{\mu} = \Xi$ for some $\mu\in\MS$, then all we need to do is apply lemma \ref{conv-lemma-psi}. \qed
\end{proof}
which also leads to the following corollary that will be useful later

\begin{corollary}
$\mathcal{D}(\mathcal{S}) = \hull{f(\mathcal{S})}$
\label{co:D-Hf}
\end{corollary}

\begin{proof}
Corollary \ref{co:existence} immediatly implies that $\phi\in\mathcal{D}(\mathcal{S}) \Rightarrow \phi\in\hull{f(\mathcal{S})}$. On the other hand, if $\phi\in\hull{f(\mathcal{S})}$, then we can use Caratheodory's theorem and a reasoning similar to the one in the proof of corollary \ref{co:existence} to find a measure $\mu$ with finite support and such that $\expect{f(X)}_{\mu} = \phi$. Since the support of $\mu$ is finite, then $\expect{g(X)}_{\mu}$ is finite, implying $\mu\in\MSF$ and hence $\phi\in\mathcal{D}(\mathcal{S})$. \qed
\end{proof}

The connection of corollary \ref{co:existence} with the problem we are interested in is given by the following theorem

\begin{theorem}[Hull Inequality]
Let $\mathcal{S} \subseteq \mathds{R}^n$, then

\[
\inf_{\GSF} z = \inf_{\mu\in\MSF} \expect{g(X)}_{\mu} = \inf_{\mu\in\MSM{m+2}} \expect{g(X)}_{\mu}
\]\[
\sup_{\GSF} z = \sup_{\mu\in\MSF} \expect{g(X)}_{\mu} = \sup_{\mu\in\MSM{m+2}} \expect{g(X)}_{\mu}
\]
\noindent
(where we recall that $z$ denotes the $(m+1)$-th coordinate of a point)
\label{hull-ineq}
\end{theorem}

\begin{proof}
Recalling that $\GSF = \hull{\Gamma(\mathcal{S})} \cap \left(\{\phi\}\times\mathds{R}\right)$, the case $\GSF=\varnothing$ is trivially true. Moving on to $\GSF\neq\varnothing$, we'll give the proof for the supremum only, as the result for the infimum would follow from considering the supremum for $-g(X)$ instead of $g(X)$. For conciseness, let us denote
\[
\sigma = \sup_{\GSF} z\quad\mbox{and}\quad s = \sup_{\mu\in\MSF} \expect{g(X)}_{\mu}
\]
We first prove that $s \geq \sigma$. There exists a sequence $\left(P_k\right)_{k=1}^{\infty}$ in $\GSF$ such that $P_k = (\phi, z_k)$ and
\[
\lim_{k\rightarrow\infty} z_k = \sigma
\]
Using corollary \ref{co:existence} it follows that for every $k$ there exists $\mu_k\in\MSF$ such that $\expect{\Gamma(X)}_{\mu_k} = P_k$, hence
\[
\lim_{k\rightarrow\infty} \expect{g(X)}_{\mu_k} = \sigma
\]
implying $s\geq \sigma$ (note that this reasoning works even if $\sigma$ is $\infty$).

Next we show that $s \leq \sigma$. If $\sigma$ is $\infty$, there is nothing to prove, otherwise suppose by absurd that $s>\sigma$. It follows that there exists $\mu\in\MSF$ such that $g \equiv \expect{g(X)}_{\mu} > \sigma$. Let $P = \expect{\Gamma(X)}_{\mu}$. By definition of $\MSF$ we have $P = (\phi, g)$. At the same time, lemma \ref{conv-lemma-psi} implies $P\in\hull{\Gamma(\mathcal{S})}$ and hence $P\in\GSF$, but then the supremum $\sigma$ should be at least $g$ (since it is the $z$ coordinate of a point in $\GSF$). Contradiction!

Finally, using corollary \ref{co:existence}, for every $\mu\in\MSF$, we have $\expect{\Gamma(X)}_{\mu} \in \hull{\Gamma(\mathcal{S})}$. So we can use Caratheodory's theorem to build a measure $\widetilde{\mu}\in \MSM{m+2}$ such that $\expect{\Gamma(X)}_{\mu} = \expect{\Gamma(X)}_{\widetilde{\mu}}$ and hence such that $\expect{g(X)}_{\mu} = \expect{g(X)}_{\widetilde{\mu}}$, which completes the proof. \qed
\end{proof}

Note that if we supplement this theorem with corollary \ref{co:existence} we get that the possible values for $\expect{g(X)}_{\mu}$ are exactly the $z$ coordinates in $\GSF$, which provides us, at least in principle, with the answer of the problem we set to study. This can be put in the form
\begin{equation}
\inf_{\GSF} z \leq \expect{g(X)}_{\mu} \leq \sup_{\GSF} z
\label{eq:hull-ineq}
\end{equation}
which can be seen as a generalization of Jensen's inequality, as we will see from examples in the following sections. Furthermore, the equality in theorem \ref{hull-ineq} means that it can always be saturated when the bounds are finite (in the sense that we can find measures such that $\expect{g(X)}$ is arbitrarily close to the bounds and in some cases achieve equality).

Also, the fact that $\MSF$ and $\MSM{m+2}$ have the same extrema allows us to think of the problem of obtaining these bounds as an optimization over $\MSM{m+2}$, which we will explore later.

However, using convex hulls still obfuscates the results due to their complexity, so in the following sections (\ref{sec:geom-approach} and \ref{sec:numerics}) we develop tools to obtain the extrema in $\GSF$ more easily.

\subsection{Some examples}
\label{sec:examples}

Theorem \ref{hull-ineq}, that we just proved, tells us that studying $\hull{\Gamma(\mathcal{S})}$ gives us information about $\expect{g(X)}$. Before moving on we consider some examples where the convex hulls are easily accessible, to make things more explicit.

\subsubsection{Possible variances for a variable in an interval}

Consider the situation where we have a random variable with support contained in $\mathcal{S} = [a,b]$ and we are interested on the question of what are the possible values for $\vari{X}$ given $\expect{X}$. This is equivalent to bounding $\expect{X^2}$ given $\expect{X} = \lambda$, so we'd have $m=n=1$, $f(x) = x$, $g(x) = x^2$ (so $\Gamma(x) = (x, x^2)$) and $\phi = \lambda$. The convex hull of $\Gamma([a,b])$ can be obtained analytically (a qualitative graph of it is in figure \ref{fig:quad-example}). A fairly easy calculation leads us to

\[
(x,y)\in \hull{\Gamma([a,b])}\quad \Leftrightarrow\quad x^2 \leq y \leq (a+b) x - ab
\]

\begin{figure}[H]
\centering
\includegraphics[width=0.3\textwidth]{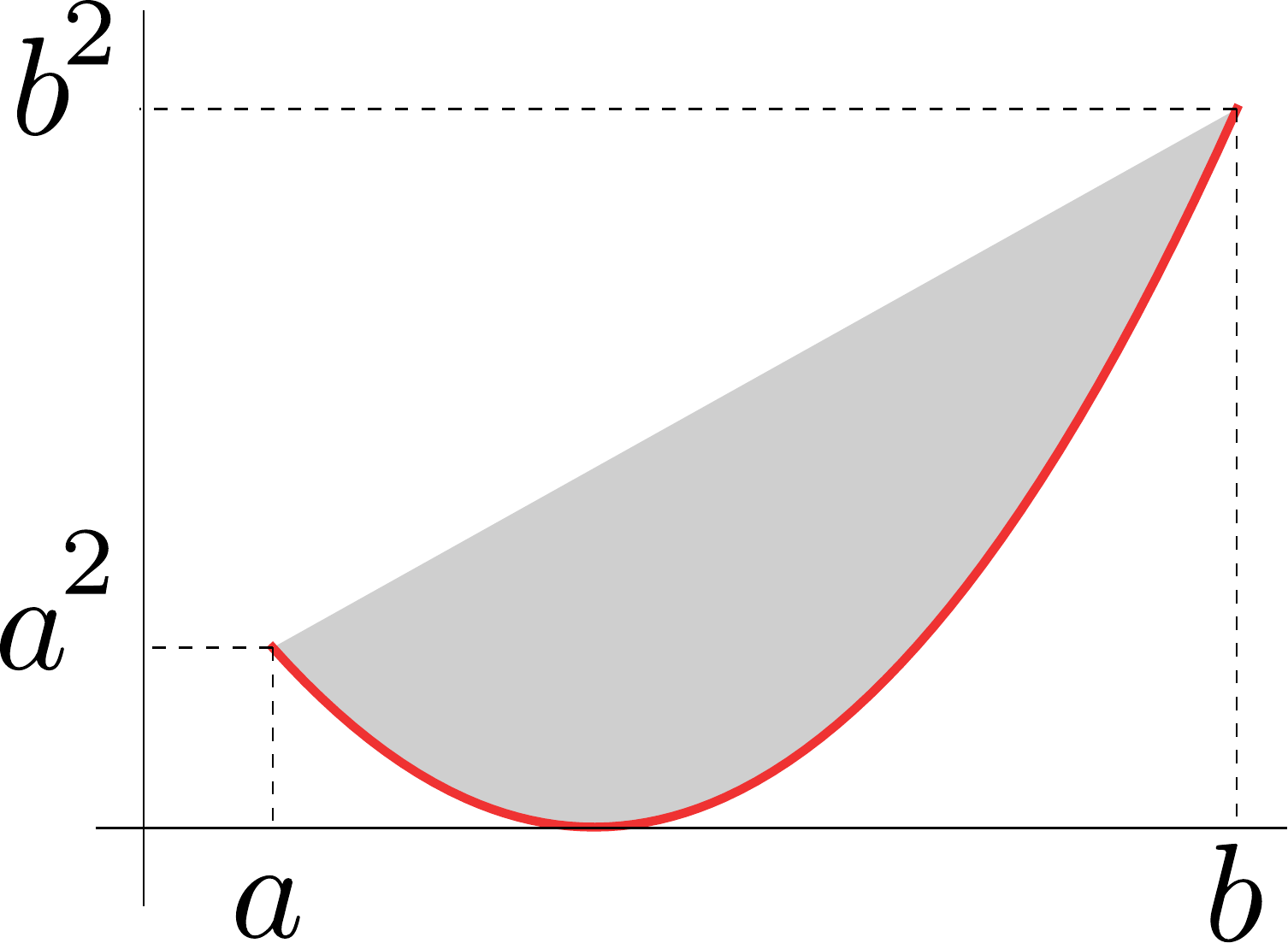}
\caption{$\Gamma(\mathcal{S})$ (in red) and its convex hull (in grey). All values of $(\expect{X}_{\mu}, \expect{X^2}_{\mu})$ for $\mu\in\MS$ lie in the hull.}
\label{fig:quad-example}
\end{figure}

Using what we just saw about the connection between convex hulls and expectations, this leads us to 
\[
\expect{X}^2 \leq \expect{X^2} \leq (a+b)\expect{X} - ab \quad \Leftrightarrow \quad \lambda^2 \leq \expect{X^2} \leq (a+b)\lambda - ab
\]
which is the bound one derives from theorem \ref{hull-ineq} in this case. Making the connection with the variance we get
\begin{equation}
0 \leq \vari{X} \leq (a+b)\lambda - \lambda^2 - ab = (b - \lambda)(\lambda - a)
\label{eq:vari-example}
\end{equation}

\subsubsection{A ``trivial'' generalization of Jensen's inequality}
\label{sec:jensen-example}

Consider the situation where we are given the expectation of the random vector $\expect{X}\in\mathds{R}^n$ and we want to study $\expect{g(X)}$. This is the case where $m=n$, $f_i(x) = x_i$, $\mathcal{S} = \mathds{R}^n$ and $\phi = \expect{X}$. In this case, $\Gamma(\mathcal{S})$ is the graph of $g(x)$:

\[
\mathcal{G} = \{(x,g(x))\,|\,x\in\mathds{R}^n\}
\]
whose convex hull obeys

\[
\hull{\Gamma(\mathcal{S})} = \hull{\mathcal{G}} \subseteq \clo{\hull{\mathcal{G}}} = \{(x,z)\,\vert\,x\in\mathds{R}^n\mbox{ and }\breve{g}(x) \leq z \leq \invbreve{g}(x)\}
\]
where $\breve{g},\invbreve{g}:\mathds{R}^n  \rightarrow\overline{\mathds{R}}$ are respectively the convex and concave envelopes of $g$.

Applying theorem \ref{hull-ineq} we get that

\begin{equation}
\breve{g}\left(\expect{X}_{\mu}\right) \leq \expect{g(X)}_{\mu} \leq \invbreve{g}\left(\expect{X}_{\mu}\right)
\label{eq:jensen-envelopes}
\end{equation}
which can be seen as a generalization of Jensen's inequality, beyond the convex/concave case. Note however that we could have proven this inequality directly from Jensen's inequality together with

\[
\breve{g}(x) \leq g\left(x\right) \leq \invbreve{g}(x)
\]
since $\breve{g}$ is convex and $\invbreve{g}$ is concave.

\section{Supporting Hyperplane Approach}
\label{sec:geom-approach}

Consider now the situation where one of the extrema
\[
\sigma_{-} = \inf_{\GSF} z\quad\mbox{or}\quad \sigma_{+} = \sup_{\GSF} z
\]
is finite. It follows that there'd exist a point $P = (\phi, \sigma)$ (either $(\phi, \sigma_-)$ or $(\phi, \sigma_+)$) in $\clo{\GSF}$. This point would be in the boundary of $\hull{\Gamma(\mathcal{S})}$ and hence there exists a supporting hyperplane passing through $P$ that divides the space in two parts, one of which has no intersection with $\Gamma(\mathcal{S})$. It turns out that the existence of these hyperplanes can help us finding $\sigma_{\pm}$ (and hence the bound in the expectation we are interested in)

\begin{theorem}
If $\mathcal{S} \subseteq \mathds{R}^n$, then there exist vectors $\alpha_{\pm} = (\alpha_{\pm}^1,$ $\ldots,$ $\alpha_{\pm}^m)$ and values $\beta_{\pm}, c_{\pm}$ such that $\beta_{\pm} \geq 0$, $(\alpha_{\pm}^1,$ $\ldots,$ $\alpha_{\pm}^m,$ $\beta_{\pm})$$\neq\vec{0}$ and

\begin{itemize}
\item If $\sigma_+$ is finite, then
\[
\innp{\alpha_+}{f(x)} + \beta_+g(x) + c_+ \leq 0\,\forall\, x\in\mathcal{S}\quad\quad\mbox{and}\quad
\sup_{\mu\in\MSF}\expect{\innp{\alpha_+}{f(X)} + \beta_+g(X) + c_+}_{\mu} = 0
\]
\item If $\sigma_-$ is finite, then
\[
\innp{\alpha_-}{f(x)} + \beta_-g(x) + c_- \geq 0\,\forall\, x\in\mathcal{S}\quad\quad\mbox{and}\quad
\inf_{\mu\in\MSF}\expect{\innp{\alpha_-}{f(X)} + \beta_-g(X) + c_-}_{\mu} = 0
\]
\item Moreover, for each case where $\sigma_{\pm}$ is finite, if $(\phi,\sigma_{\pm})\in\GSF$, then there exists a measure $\mu_{\pm}\in\MSM{m+1}$ with support $s_{\pm}$, such that
\[
\expect{g(X)}_{\mu_{\pm}} = \sigma_{\pm}\quad\mbox{and}\quad \innp{\alpha_{\pm}}{f(x)} + \beta_{\pm}g(x) + c_{\pm} = 0\,\forall\, x\in s_{\pm}
\]
\end{itemize}
\label{th:analytical}
\end{theorem}

\begin{proof}
We will focus on the proof for $\sigma_+$ (the proof for $\sigma_-$ is analogous).

Let us prove the first item, starting with the inequality
\begin{equation}
\innp{\alpha_+}{f(x)} + \beta_+g(x) + c_+ \leq 0\,\forall\, x\in\mathcal{S}
\label{eq:ineq-proof}
\end{equation}

Since $\sigma_+$ is finite, then using theorem \ref{hull-ineq}, there exists a point $P_+ = (\phi, \sigma_+) \in \clo{\GSF}$. This point must be in the boundary of $\hull{\Gamma(\mathcal{S})}$, so there must exist a supporting hyperplane $\Pi$ containing it. This hyperplane divides the space into 2 regions, one of which has no intersection with $\hull{\Gamma(\mathcal{S})}$ (and hence with $\Gamma(\mathcal{S})$).

Algebraically, the hyperplane and the regions it defines can be specified (for every $y\in\mathds{R}^{m+1}$) as $\innp{a_+}{y} + c_+ = 0$ ($\Pi$ itself) and $\innp{a_+}{y} + c_+ > 0$, $\innp{a_+}{y} + c_+ < 0$ (the 2 regions), where $a_+\equiv (\alpha_+^1, \ldots, \alpha_+^m, \beta_+)\neq \vec{0}$. Note that the sign of $a_+$ is arbitrary (as we can change it and the sign of $c_+$ to get the same geometric locus) so without loss of generality we may assume $\beta_+ \geq 0$.

We will now consider some possibilities. If $\Gamma(\mathcal{S})$ is such that $\inter{\hull{\Gamma(\mathcal{S})}} = \varnothing$, then $\Pi$ can also be made such that $\hull{\Gamma(\mathcal{S})} \subseteq \Pi$, meaning that equality holds in (\ref{eq:ineq-proof}) for all points $x\in\mathcal{S}$ (and we are done). We will divide the case $\inter{\hull{\Gamma(\mathcal{S})}} \neq \varnothing$ in two other cases. Firstly, if $\beta_+=0$, we need to prove that we can choose $a_+$ and $c_+$ such that

\[
\innp{a_+}{\Gamma(x)} + c_+ = \innp{\alpha_+}{f(x)} + c_+ \leq 0\,\forall\, x\in\mathcal{S}
\]
But since $\beta_+=0$, this implies that the sign of $a_+$ has still not been determined. So the observation that $\innp{a_+}{y} + c_+$ has the same sign for all $y\in \Gamma(\mathcal{S})$ trivially implies that we can choose this sign to be negative.

Finally, for the case $\beta_+ > 0$, consider the line $\lambda = \{(\phi, t) \,|\, t\in\mathds{R}\}$. If $\lambda \cap \inter{\hull{\Gamma(\mathcal{S})}} = \varnothing$ then the separating hyperplane theorem tells us that there exists a hyperplane $\Xi$ separating $\lambda$ and $\inter{\hull{\Gamma(\mathcal{S})}}$. Since $P_+\in\lambda$, the distance between these two sets is 0, implying that $\Xi$ is also a supporting hyperplane containing $P_+$. Furthermore, it implies that $\lambda \subseteq \Xi$, so if we had used $\Xi$ instead of $\Pi$ to define $\alpha$ and $\beta$, we'd be back to the case $\beta_+ = 0$ that we proved already. On the other hand, if $\lambda$ and $\inter{\hull{\Gamma(\mathcal{S})}}$ are not disjoint, then $\lambda \cap \hull{\Gamma(\mathcal{S})}$ is not a singleton, implying that there are points of the form $(\phi, t)$ in $\hull{\Gamma(\mathcal{S})}$ with $t \neq \sigma_+$. These points must have $t < \sigma_+$ and hence the region defined by $\Pi$ where $\hull{\Gamma(\mathcal{S})}$ (and hence $\Gamma(\mathcal{S})$) resides is $\innp{a_+}{y} + c_+ \leq 0$. Therefore, $\innp{a_+}{\Gamma(x)} + c_+ \leq 0\,\forall\,x\in\mathcal{S}$. Expanding this we get the desired inequality:

\[
\innp{\alpha_+}{f(x)} + \beta_+g(x) + c_+ \leq 0\,\forall\, x\in \mathcal{S}
\]
(in the proof for $\sigma_-$ the only differences are that when $\beta=0$, we choose the sign to be positive and on the last step we have $t > \sigma_-$, which reverses the sign of the final inequality).

Next, we must prove that

\[
\sup_{\mu\in\MSF}\expect{\innp{\alpha_+}{f(x)} + \beta_+g(x) + c_+}_{\mu} = 0
\]
We first note that if $\mu\in\MSF$, then

\[
\expect{\innp{\alpha_+}{f(x)} + \beta_+g(x) + c_+}_{\mu} = \innp{\alpha_+}{\phi} + c_+ + \beta_+\expect{g(x)}_{\mu}
\]
Since $\beta_+ \geq 0$, this implies that

\[
\sup_{\mu\in\MSF}\expect{\innp{\alpha_+}{f(x)} + \beta_+g(x) + c_+}_{\mu} = \innp{\alpha_+}{\phi} + c_+ + \beta_+\sigma_+ = \innp{a_+}{P_+} + c_+
\]
which equals 0 because $P_+ \in \Pi$ (the reasoning is identical for $\sigma_-$).

Finally, let us prove the last item. Since $P_+ \in \GSF$, then $P_+ \in \hull{\Gamma(\mathcal{S})}$ and we can use Caratheodory's theorem to write $P_+$ as the convex combination of $m+1$ points in $\Gamma(\mathcal{S})$ ($m+2$ are not needed because we are in the boundary) and using an argument similar to the one in the proof of corollary \ref{co:existence}, there exists $\mu_+\in\MSM{m+1}$ with $\expect{\Gamma(X)}_{\mu_+} = P_+$. Hence $\expect{g(X)}_{\mu_+} = \sigma_+$.

Also,
\[
\expect{\innp{a_+}{\Gamma(X)} + c_+}_{\mu_+} = \innp{a_+}{\expect{\Gamma(X)}_{\mu_+}} + c_+ = \innp{a_+}{P_+} + c_+ = 0.
\]
Since $\innp{a_+}{\Gamma(x)} + c_+$ must have the same sign for all $x\in\mathcal{S}$ and its expectation is null, then $\mu_+$ must be such that $\innp{a_+}{\Gamma(X)} + c_+ = 0$ almost surely. Since the support $s_+$ of $\mu_+$ is finite, then $s_+ \subseteq \Pi$. Expanding $\innp{a_+}{\Gamma(x)}$ we get

\[
\innp{\alpha_+}{f(x)} + \beta_+g(x) + c_+ = 0\,\forall\, x\in s_+
\]
(once again the reasoning is identical for $\sigma_-$ and $\mu_-$) \qed
\end{proof}

An important situation where this theorem can be applied is when $\mathcal{S}$ is a compact in $\mathds{R}^n$ and the restriction of $\Gamma$ to $\mathcal{S}$ is continuous. In this case both $\sigma_{\pm}$ are finite and such that $(\phi, \sigma_{\pm})\in\GSF$ and as we will see in section \ref{ssec:examples}, studying the possibilities for $\mu_{\pm}$ will be very useful for finding the $\sigma_{\pm}$.
The next theorem will allow us to extend this use case, by studying progressive compact covers of $\mathcal{S}$, instead of $\mathcal{S}$ itself.

\begin{theorem}
Let $\left(\mathcal{S}_i\right)_{i=1}^{\infty}$ be a progressive cover of $\mathcal{S} \subseteq \mathds{R}^n$ and let $L_i$, $U_i$, $L_i^{(k)}$ and $U_i^{(k)}$ be defined as follows

\[
L_i\equiv\inf_{\mu\in M(\mathcal{S}_i,\phi)} \expect{g(X)}_{\mu}\quad\quad\quad U_i\equiv\sup_{\mu\in M(\mathcal{S}_i,\phi)} \expect{g(X)}_{\mu}
\]\[
L_i^{(k)}\equiv\inf_{\mu\in M_k(\mathcal{S}_i,\phi)} \expect{g(X)}_{\mu}\quad\mbox{and}\quad U_i^{(k)}\equiv\sup_{\mu\in M_k(\mathcal{S}_i,\phi)} \expect{g(X)}_{\mu}
\]
then

\[
\inf_{\mu\in \MSF} \expect{g(X)}_{\mu} = \lim_{i\rightarrow\infty} L_i \quad\quad\quad
\sup_{\mu\in \MSF} \expect{g(X)}_{\mu} = \lim_{i\rightarrow\infty} U_i
\]\[
\inf_{\mu\in \MSM{k}} \expect{g(X)}_{\mu} = \lim_{i\rightarrow\infty} L_i^{(k)} \quad\mbox{and}\quad
\sup_{\mu\in \MSM{k}} \expect{g(X)}_{\mu} = \lim_{i\rightarrow\infty} U_i^{(k)}
\]
\label{th:lim}
\end{theorem}

\begin{proof}
Let us show that 
\[
\sup_{\mu\in \MSF} \expect{g(X)}_{\mu} = \lim_{i\rightarrow\infty} U_i
\]

Since $\left(\mathcal{S}_i\right)_{i=1}^{\infty}$ is a progressive cover, it follows that $M(\mathcal{S}_i,\phi) \subseteq M(\mathcal{S}_j,\phi) \subseteq \MSF$ whenever $i<j$. Firstly, this implies that the case $\MSF = \varnothing$ is such that $M(\mathcal{S}_i,\phi) = \varnothing\,\forall\, i$, hence the supremum and all the $U_i$ are $-\infty$ (proving this case). Secondly, if $\MSF \neq \varnothing$, then $\left(U_i\right)_{i=1}^{\infty}$ is non-decreasing.

Consider first the case where the supremum is finite:
\[
\sup_{\mu\in \MSF} \expect{g(X)}_{\mu} = \sigma
\]
It follows that for every $\varepsilon > 0$, there exists $\mu_{\varepsilon}\in\MSF$ such that

\[
\overline{g} = \expect{g(X)}_{\mu_{\varepsilon}} \geq \sigma - \varepsilon
\]
We must have then $\expect{\Gamma(X)}_{\mu_{\varepsilon}} = (\phi,\overline{g})$. Using lemma \ref{conv-lemma-psi} we have $(\phi,\overline{g}) \in \hull{\Gamma(\mathcal{S})}$. Using Caratheodory's theorem we can build then $\nu_{\varepsilon}\in\MSM{m+2}$ with $\expect{\Gamma(X)}_{\nu_{\varepsilon}} = (\phi,\overline{g})$. Let $s_{\varepsilon}$ be the support of $\nu_{\varepsilon}$. Since $s_{\varepsilon}\subseteq \mathcal{S}$ is finite and $\left(\mathcal{S}_i\right)_{i=1}^{\infty}$ is a progressive cover of $\mathcal{S}$, then there exists $N_{\varepsilon}$ such that for all $n > N_{\varepsilon}$ we have $s_{\varepsilon}\subseteq \mathcal{S}_n$, implying $\nu_{\varepsilon}\in M(\mathcal{S}_n,\phi)$ and hence

\[
U_n = \sup_{\mu\in M(\mathcal{S}_n,\phi)} \expect{g(X)}_{\mu} \geq \expect{g(X)}_{\nu_{\varepsilon}} = \overline{g}
\]
However, since $M(\mathcal{S}_n,\phi) \subseteq \MSF$, then $U_n\leq \sigma$, so $|U_n - \sigma| \leq \varepsilon$, implying the limit.

The case when the supremum is $\infty$ is similar. For all $\delta$, there exists $\mu_{\delta}\in\MSF$ such that 

\[
\expect{g(X)}_{\mu_{\delta}} \geq \delta
\]
Invoking lemma \ref{conv-lemma-psi} and Caratheodory's theorem we can once again find $\nu_{\delta}\in\MSM{m+2}$ with the same expectation $\expect{\Gamma(X)}$ than $\mu_{\delta}$. Once again, since the support of $\nu_{\delta}$ is finite, then there exists $N_{\delta}$ such that for all $n > N_{\delta}$, the support of $\nu_{\delta}$ is in $\mathcal{S}_n$, implying $U_n \geq \delta$ and hence that $U_n\rightarrow\infty$.

This concludes the proof, as the limit for the infimum follows from considering the limit of the supremum for $-g$ instead of $g$ and the reasoning for the case with finite support is nearly identical (substitute $M$ for $M_k$, $U_i$ for $U_i^{(k)}$ and instead of using lemma 2 together with Caratheodory's theorem to build $\nu_{\varepsilon}$ and $\nu_{\delta}$, we can just use $\mu_{\varepsilon}$ and $\mu_{\delta}$ in their places, as the supports are already finite). \qed
\end{proof}

\subsection{Some simple examples}
\label{ssec:examples}

To illustrate how to use theorems \ref{th:analytical} and \ref{th:lim}, we will first work out some examples where the bounds can also be derived by simpler methods.

\subsubsection{A case where $\Gamma$ is continuous}

Let us first find the lower bound for $\expect{X^4}$ given $\expect{X}=\lambda$ and $\vari{X}=\sigma^2>0$. This problem fits the framework we are developping. Namely we have $m=2$, $f_1(x) = x$, $f_2(x) = x^2$, $g(x) = x^4$, $\phi = (\lambda, \lambda^2 + \sigma^2)$ and $\mathcal{S} = \mathds{R}$. Theorem \ref{th:analytical} gives us the most information when $\Gamma(\mathcal{S})$ is compact, which is not the case. However, we can use theorem \ref{th:lim} and study a progressive compact cover of $\mathds{R}$ instead. We will consider a cover where all elements are intervals of the form $[-a,a]$ (what the particular cover is turns out to be unimportant). Since $\Gamma([-a,a])$ is compact, then the corresponding $\gamma([-a,a],\phi)$ will be compact and theorem 2 implies that there exists a measure in $M_3([-a,a], \phi)$ that attains the infimum of $\expect{X^4}_{\mu}$ for $\mu\in M([-a,a], \phi)$. The support of this measure consists of roots of $h(x)$, where

\[
h(x) \equiv \alpha_0 + \alpha_1 x + \alpha_2 x^2 + \beta x^4 \geq 0\,\,\forall\,\,x\in [-a,a]
\]
for some choice of constants $\beta, \alpha_k$, with $\beta \geq 0$. We note that the roots in $]-a,a[$ must be double roots, while $\pm a$ may be simple roots. As $\sigma > 0$ then the support must have more than one point and hence $h(x)$ must have more than one root. These constraints leave us with the possibilities found in figure \ref{fig:x4}, for the qualitative graph of $h(x)$.

\begin{figure}[hbtp]
\centering
\subfigure[]{\includegraphics[width=0.2\textwidth]{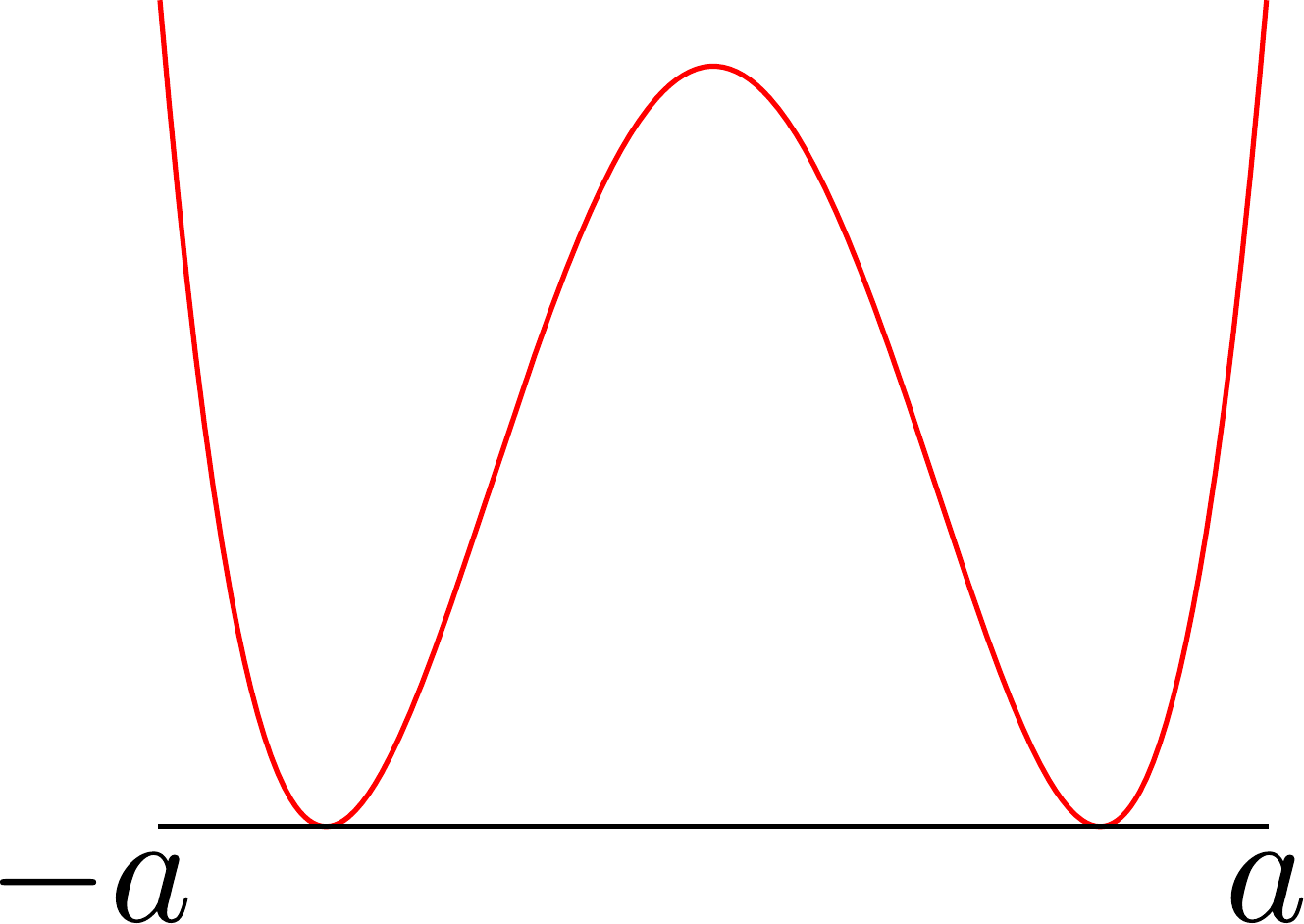}}\quad\quad
\subfigure[]{\includegraphics[width=0.2\textwidth]{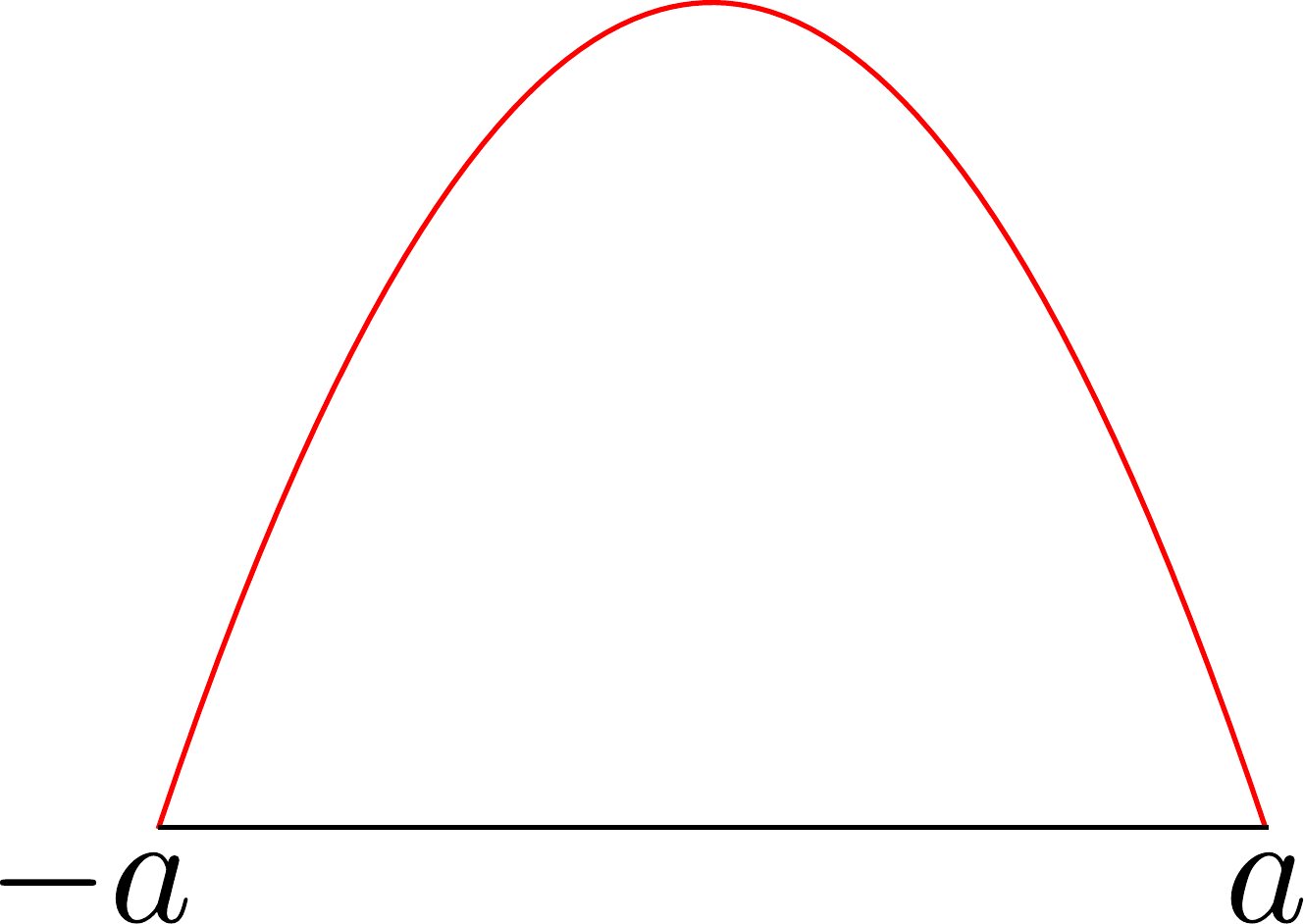}}\quad\quad
\subfigure[]{\includegraphics[width=0.2\textwidth]{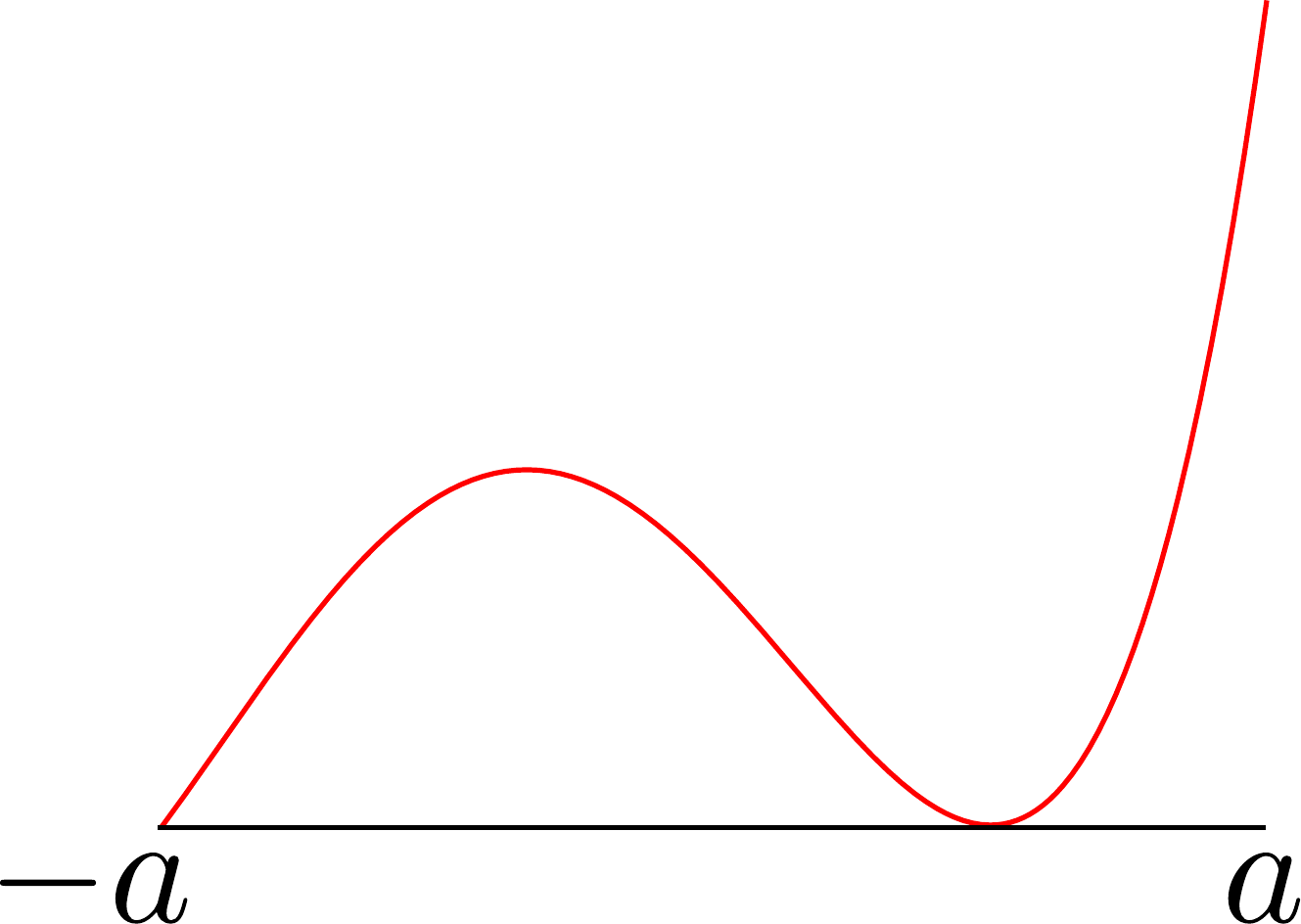}}\quad\quad
\subfigure[]{\includegraphics[width=0.2\textwidth]{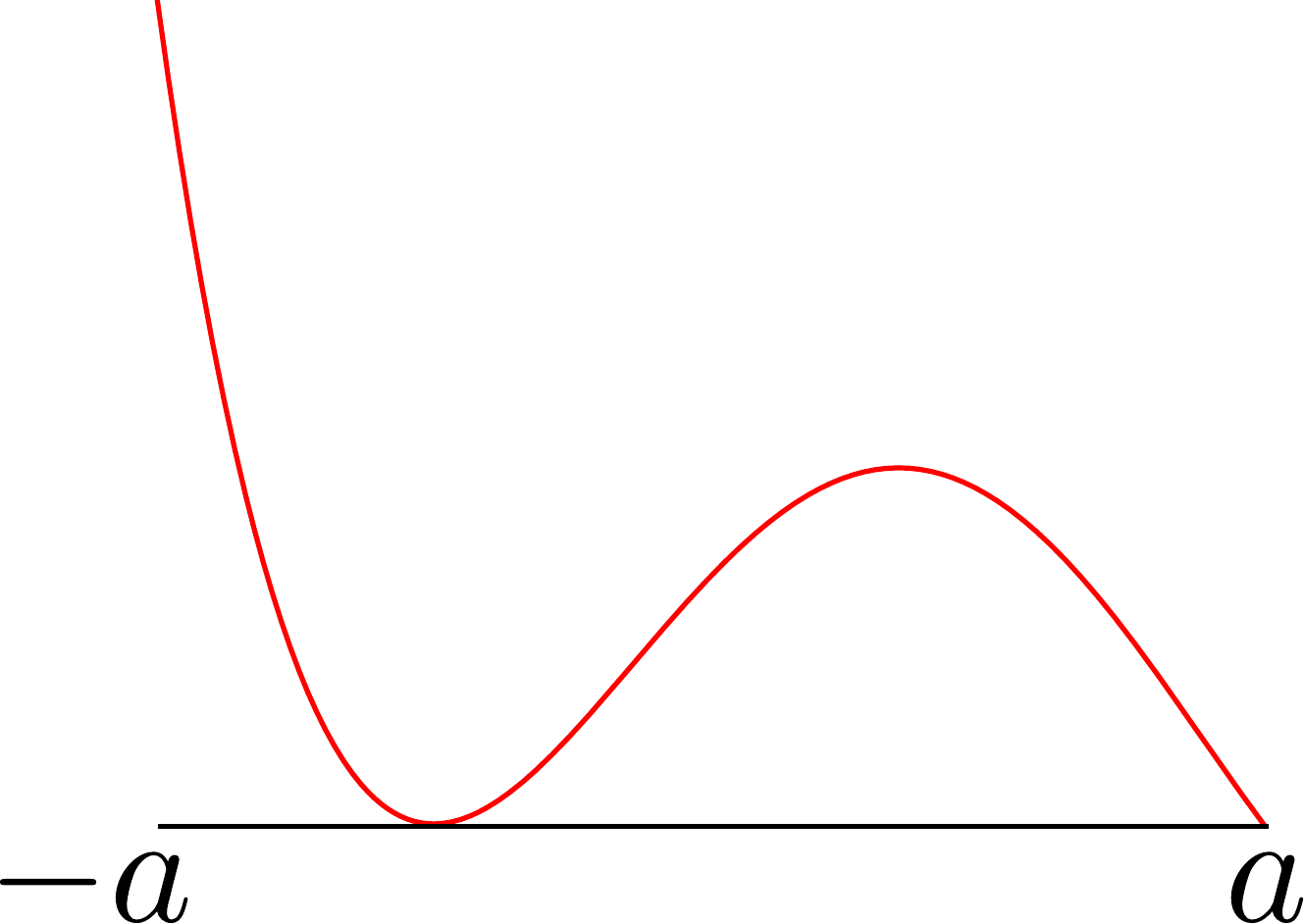}}
\caption{Possibilities for the qualitative graph of $h(x)$ respecting $h(x)\geq 0$ in $[-a,a]$ and with at least 2 roots in the interval.}
\label{fig:x4}
\end{figure}

Adding the constraint that the cubic term in $h(x)$ must be 0 allows us to discard the possibilities in \ref{fig:x4}(c) and (d), as the sum of the roots is necessarily different from zero in these cases (in \ref{fig:x4}(c), for example, the roots are $b < a$ with multiplicity 2, $-a$ with multiplicity 1 and $c \leq -a$ with multiplicity 1, so their sum taking the multiplicities into account must be negative). The sum of the roots being zero also implies that in the cases \ref{fig:x4}(a) and (b) the roots in the interval must be $\pm b$ for some $b \in [-a,a]$.

Imposing the known expectations, we get that $b = \sqrt{\lambda^2 + \sigma^2}$ and hence there exists only one measure obeying all the properties prescribed by theorem \ref{th:analytical} for every $a\geq b$ (the probabilities for $-b$ and $b$ are uniquely determined by $\lambda$), so it must be the one with the smallest possible expectation for $X^4$. Since this measure is the same for every $a\geq b$, the limit prescribed by theorem \ref{th:lim} is trivial and we have the bound

\[
\expect{X^4} \geq (\lambda^2 + \sigma^2)^2
\]
which can also be derived directly from Jensen's inequality

\[
\expect{(X^2)^2} \geq \expect{X^2}^2
\]

\subsubsection{Dealing with discontinuities (Markov's inequality)}

For the second example, lets rederive Markov's inequality. That is, if $X$ is a nonnegative random variable and $a>0$, then

\[
\prob{X > a} \leq \min\left\{\frac{\expect{X}}{a}, 1\right\}
\]
We can translate this into our framework, using $\mathcal{S} = \mathds{R}_+$,  $m=1$, $f_1(x) = x$, $\phi = \lambda$, $g(x) = \Theta(a-x)$, where $\Theta$ is the Heaviside step function (as defined in (\ref{eq:heaviside})) and we want to show that the lower bound for $\expect{g(X)} = \expect{\Theta(a-X)} = \prob{X \leq a}$ is $\max\left\{1 - \nicefrac{\lambda}{a}, 0\right\}$.

Once again, $\Gamma(\mathcal{S})$ is not compact, but we can use a progressive compact cover for $\mathcal{S}$. This cover must be built more carefully than in the previous example, to deal with the discontinuity of $g(x)$ at $x=a$. A possible choice is to use $\mathcal{S}_{n} = [0,a]\,\,\cup\,\, \left[a+\frac{1}{n}, n\right]$ for $n>a+1$ as the elements of the cover. Since the restriction of $\Gamma$ to any $\mathcal{S}_n$ is continuous, then $\Gamma\left(\mathcal{S}_{n}\right)$ is compact and we can once again use theorem \ref{th:analytical} to find a measure that minimizes $\expect{g(X)}$.

The support of the measure that minimizes $\expect{g(X)}$ is composed of roots of $h(x)$, where

\[
h(x) = \beta \Theta(a-x) + \alpha_1 x + \alpha_0 \geq 0
\]
for some choice of $\alpha_0, \alpha_1$ and $\beta$, such that $\beta \geq 0$. What the possible roots are will depend only on the sign of $\alpha_1$. Some representative graphs can be found in Fig \ref{fig:markov}. Taking $n > \lambda, a+1$; we have the following cases: 

\begin{description}
\item[(A)] If $\alpha_1 < 0$ then the only possible root is $n$ (Fig \ref{fig:markov}(a))
\item[(B)] If $\alpha_1 > 0$ then the only possible roots are $0$ and $a+\frac{1}{n}$ (Figs \ref{fig:markov}(b-d))
\item[(C)] If $\alpha_1 = 0$ then all values in the interval $\left[a+\frac{1}{n}, n\right]$ can be roots, as long as we also have $\alpha_0 = 0$ (Fig \ref{fig:markov}(e))
\end{description}

\begin{figure}[htb]
\centering
\subfigure[]{\includegraphics[width=0.18\textwidth]{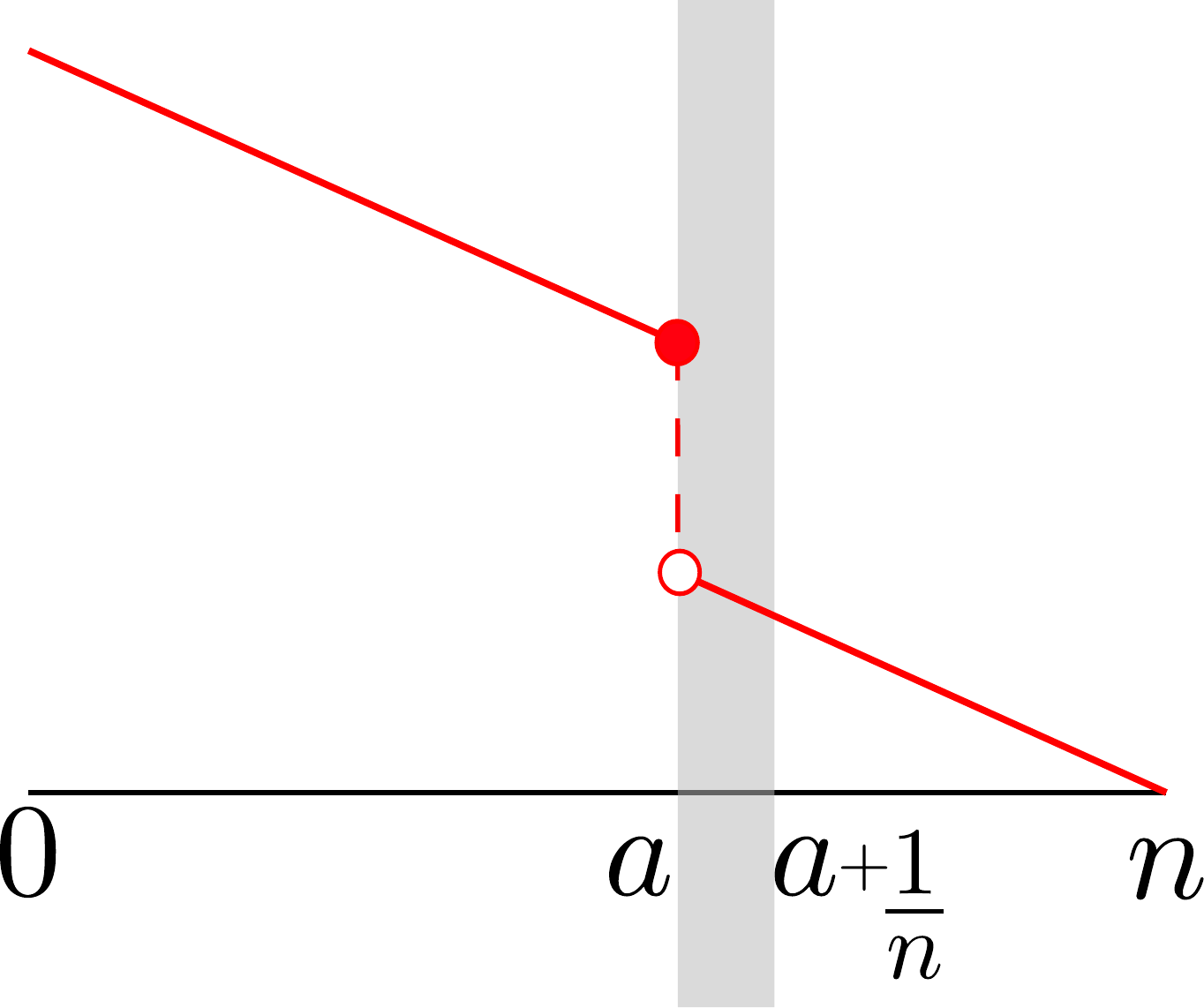}}\quad
\subfigure[]{\includegraphics[width=0.18\textwidth]{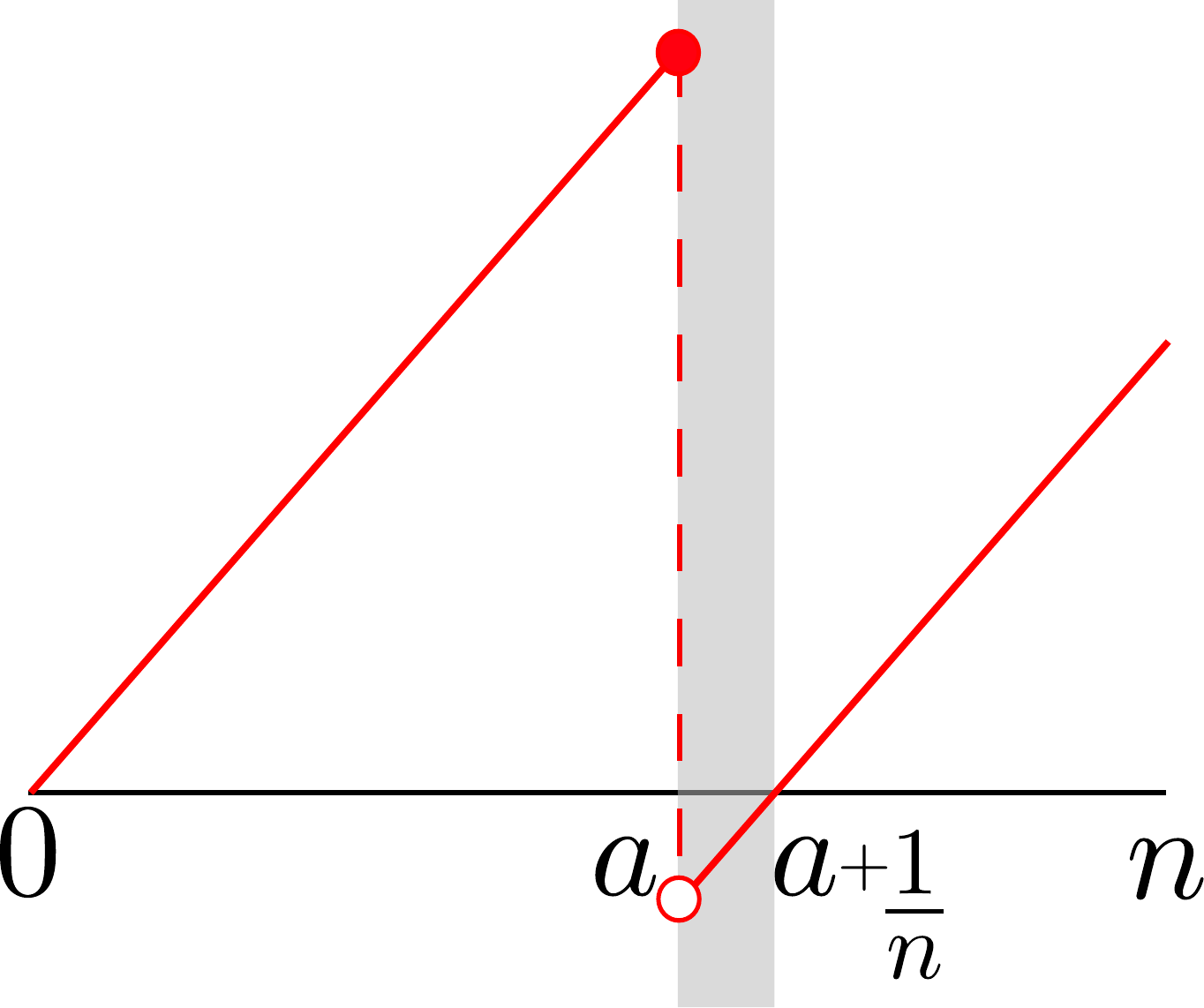}}\quad
\subfigure[]{\includegraphics[width=0.18\textwidth]{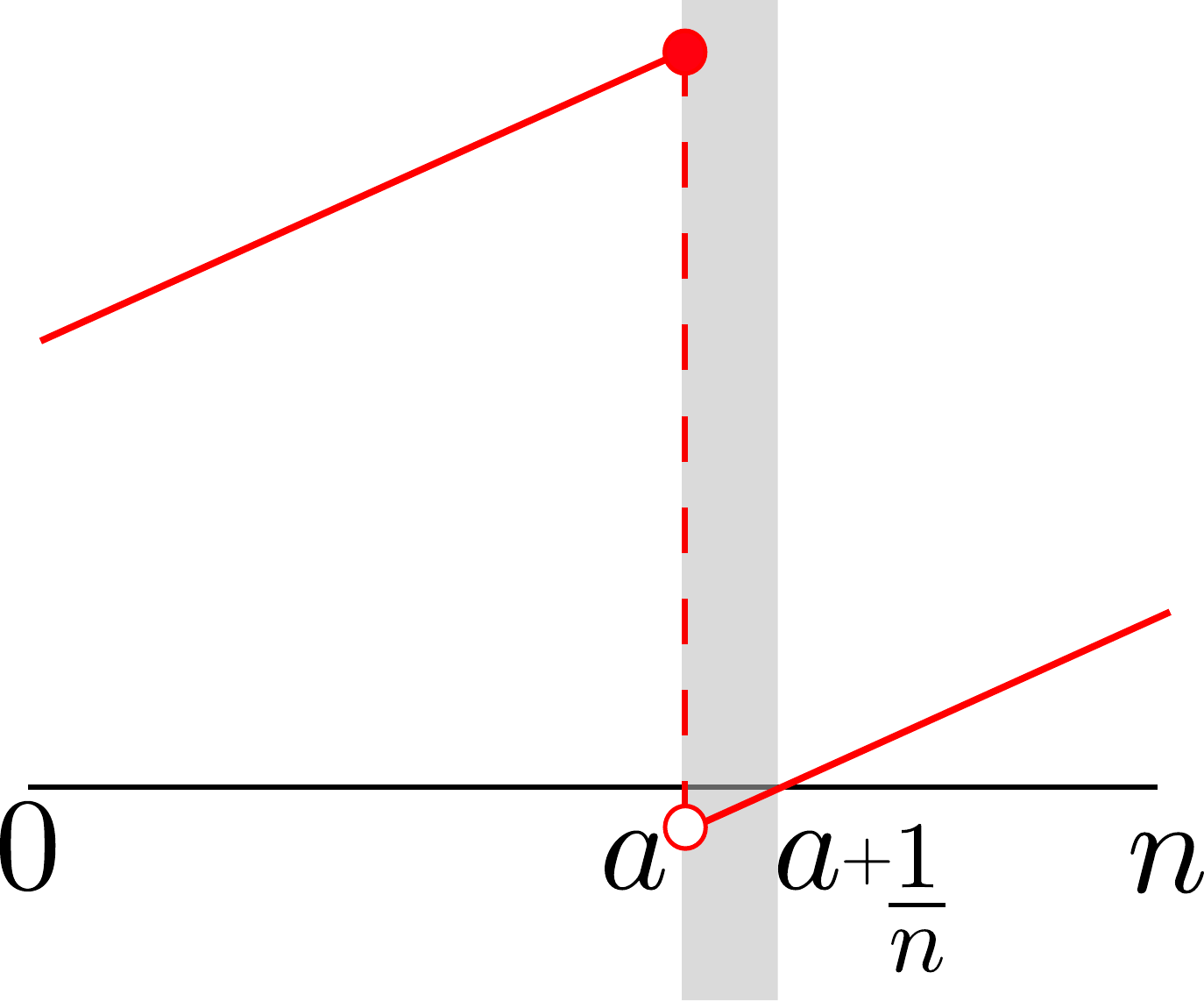}}\quad
\subfigure[]{\includegraphics[width=0.18\textwidth]{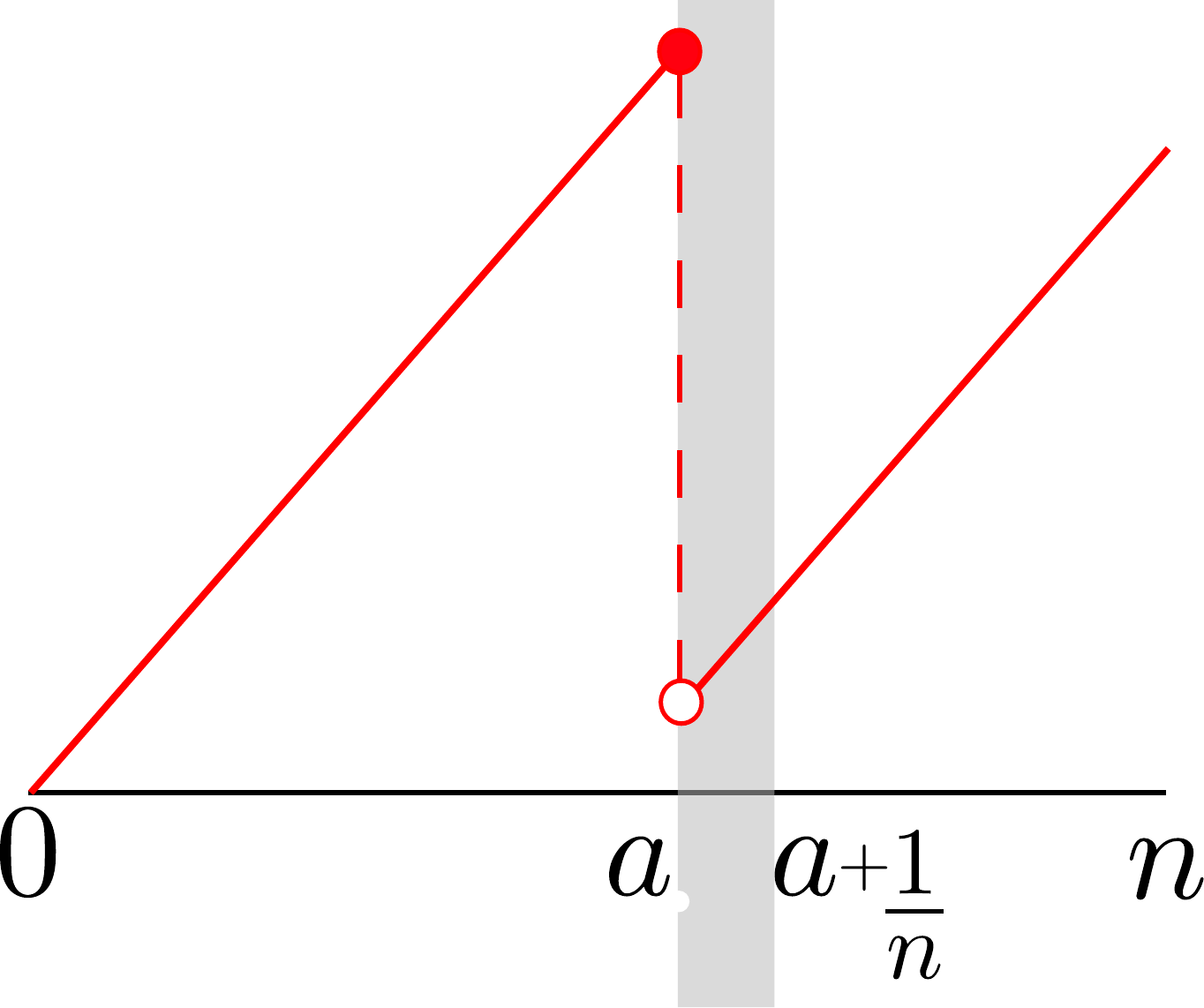}}\quad
\subfigure[]{\includegraphics[width=0.18\textwidth]{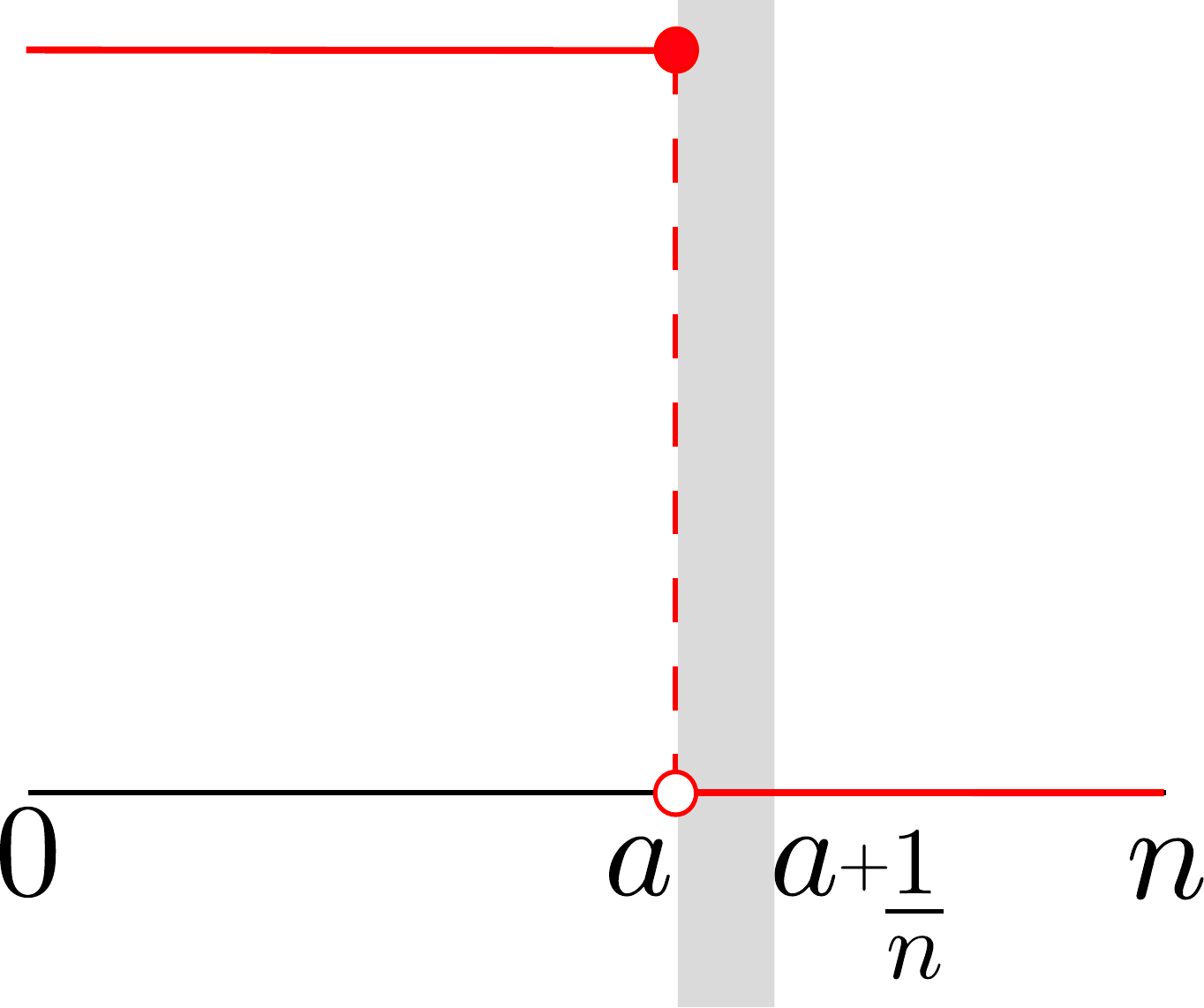}}
\caption{Possibilities for the graph of $h(x)$ respecting $h(x)\geq 0$ in $\mathcal{S}_n=[0,a]\cup [a+\nicefrac{1}{n},n]$ (the interval except the grayed out region).}
\label{fig:markov}
\end{figure}
Since $n > \lambda$, then case {\bf (A)} will be irrelevant, as it never obeys the constraints. If $\lambda \leq a$, then case {\bf (B)} is the only one that can obey the constraints and doing the algebra leads us to $\prob{X \leq a} = \expect{\Theta(a-X)} = 1 - \frac{n\lambda}{n a + 1}$. Finally, if $\lambda > a$, then case {\bf (C)} is the only one that can obey the constraints (which it does for $n>\nicefrac{1}{(\lambda - a})$) and we have trivially $\prob{X \leq a} = \expect{\Theta(a-X)} = 0$.

To obtain Markov's inequality we must use theorem \ref{th:lim} and take the limit $n\rightarrow\infty$, which is clearly

\[
\left.\lim_{n\rightarrow\infty} \expect{\Theta(a-X)}\right|_{\mathcal{S}_n} = \max\left\{1 - \frac{\lambda}{a}, 0\right\}
\]
completing the derivation.

\subsection{A note about the compact case}

The examples in section \ref{ssec:examples} highlight the importance of the case when $\mathcal{S} \subseteq \mathds{R}^n$ has a progressive compact cover. This raises the question of how to identify these situations, which fortunately is an easy one:

\begin{lemma}
$\mathcal{S} \subseteq \mathds{R}^n$ has a progressive compact cover iff it is an $F_{\sigma}$ set.
\label{lemma:Fs}
\end{lemma}

\begin{proof}
If $\mathcal{S}$ has a progressive compact cover $(\mathcal{S}_k)_{k=1}^{\infty}$, then clearly $\mathcal{S}$ is $F_{\sigma}$, as

\[
\mathcal{S} = \bigcup_{k=1}^{\infty} \mathcal{S}_k
\]
If on the other hand $\mathcal{S}$ is an $F_{\sigma}$ set, then $\mathcal{S}$ can be written as

\[
\mathcal{S} = \bigcup_{k=1}^{\infty} F_k
\]
where the $F_k$ are all closed. Let $(R_k)_{k=1}^{\infty}$ be a progressive compact cover of $\mathds{R}^n$ (like $R_k = [-k, k]^n$, for example), then if we define

\[
\mathcal{S}_k = R_k \cap \bigcup_{q=1}^{k} F_q
\]
then $(\mathcal{S}_k)_{k=1}^{\infty}$ is a progressive compact cover of $\mathcal{S}$. \qed
\end{proof}
Which leads us to the following strengthening of theorem \ref{hull-ineq}

\begin{corollary}
If $\mathcal{S}$ is an $F_{\sigma}$ set in $\mathds{R}^n$ and $\Gamma$ is continuous, then
\[
\inf_{\mu\in \MSF} \expect{g(X)}_{\mu} = \inf_{\mu\in \MSM{m+1}} \expect{g(X)}_{\mu}
\]\[
\sup_{\mu\in \MSF} \expect{g(X)}_{\mu} = \sup_{\mu\in \MSM{m+1}} \expect{g(X)}_{\mu}
\]
\label{cor:Fs}
\end{corollary}

\begin{proof}
This is trivially true if $\MSF=\varnothing$. Otherwise, lemma \ref{lemma:Fs} implies that $\mathcal{S}$ has a progressive compact cover $(\mathcal{S}_k)_{k=1}^{\infty}$. Since $\Gamma$ is continuous, then $\gamma(\mathcal{S}_k, \phi)$ is always compact and using theorems \ref{th:analytical} and \ref{th:lim} we get
\[
\inf_{\mu\in \MSF} \expect{g(X)}_{\mu} = \lim_{k \rightarrow \infty}  \left(\inf_{\mu\in M(\mathcal{S}_k,\phi)} \expect{g(X)}_{\mu}\right) = 
\]\[
= \lim_{k \rightarrow \infty} \left(\inf_{\mu\in M_{m+1}(\mathcal{S}_k,\phi)} \expect{g(X)}_{\mu}\right) = \inf_{\mu\in \MSM{m+1}} \expect{g(X)}_{\mu}
\]
The proof for the supremum follows from considering the infimum for $-g$ instead of $g$. \qed
\end{proof}

\section{Application to the Jensen gap problem}
\label{sec:jensen-gap}

The examples in the previous section were meant to familiarize the reader with this method of obtaining bounds (study how the roots can be distributed, then apply the constraints to find candidates for the measure extremizing the expectation we are interested in), by presenting situations where the results could also be obtained by more familiar methods. In the next sections we use theorems \ref{th:analytical} and \ref{th:lim} to obtain novel contributions to the problem of finding bounds for the Jensen gap. In particular we will be investigating bounds for $\expect{g(X)}$ in the case where $\expect{X}$ and $\vari{X}$ are given.

\subsection{If $g'(x)$ is strictly convex}
\label{sec:jensen-gap-conv}

\begin{theorem}
Let $X$ be a random variable with support contained in $[a,b]$ and let $g:[a,b] \rightarrow \mathds{R}$ be bounded, differentiable and such that $g'(x)$ is strictly convex. Then for every $\lambda$ and $\sigma^2 > 0$ that are possible values for the average and variance of a variable in $[a,b]$ (which amounts to $\lambda\in\,]a,b[$ and $\sigma^2 \leq (b-\lambda)(\lambda-a)$), there exist measures $\mu_{\pm}$ with
\[
\expect{X}_{\mu_{\pm}} = \lambda\quad\quad\mbox{and}\quad\quad \vari{X}_{\mu_{\pm}} = \sigma^2
\]
such that
\[
\expect{g(X)}_{\mu_-} = \frac{\sigma^2 g(a) + (\lambda - a)^2 g\left(\lambda + \frac{\sigma^2}{\lambda - a}\right)}{\sigma^2 + (\lambda - a)^2}
\]\[
\expect{g(X)}_{\mu_+} = \frac{\sigma^2 g(b) + (\lambda - b)^2 g\left(\lambda + \frac{\sigma^2}{\lambda - b}\right)}{\sigma^2 + (\lambda - b)^2}
\]
and for every measure $\mu$ in $M([a,b])$, with the same average and variance, we have
\[
\expect{g(X)}_{\mu_-} \leq \expect{g(X)}_{\mu} \leq \expect{g(X)}_{\mu_+}
\]
\label{th:dg-strict-convex}
\end{theorem}

\begin{proof}
We will focus on the lower bound, as the proof for the upper bound is analogous. We have $\mathcal{S} = [a,b]$, $f_1(x) = x$, $f_2(x) = x^2$ and $\phi=(\lambda, \lambda^2 + \sigma^2)$. Since $\Gamma$ is continuous and $\mathcal{S}$ is compact, then theorem \ref{th:analytical} implies that there exists a measure $\mu \in M_{3}([a,b],\phi)$ that minimizes $\expect{g(X)}_{\mu}$. The support of $\mu$ consists of roots of $h(x) \equiv \beta g(x) + \alpha_0 + \alpha_1 x + \alpha_2 x^2$ for some choice of $\alpha_0, \alpha_1, \alpha_2$ and $\beta$, such that $\beta \geq 0$ and $h(x) \geq 0$ in $[a,b]$. Furthermore, $\sigma > 0$, so there must be more than one point in the support. Using that $g'$ is strictly convex and $\beta \geq 0$, we can obtain all possibilities for the qualitative graph of $h(x)$ (Figs \ref{fig:proof-dsc}(a, b))

\begin{figure}[H]
\centering
\subfigure[]{\includegraphics[width=0.2\textwidth]{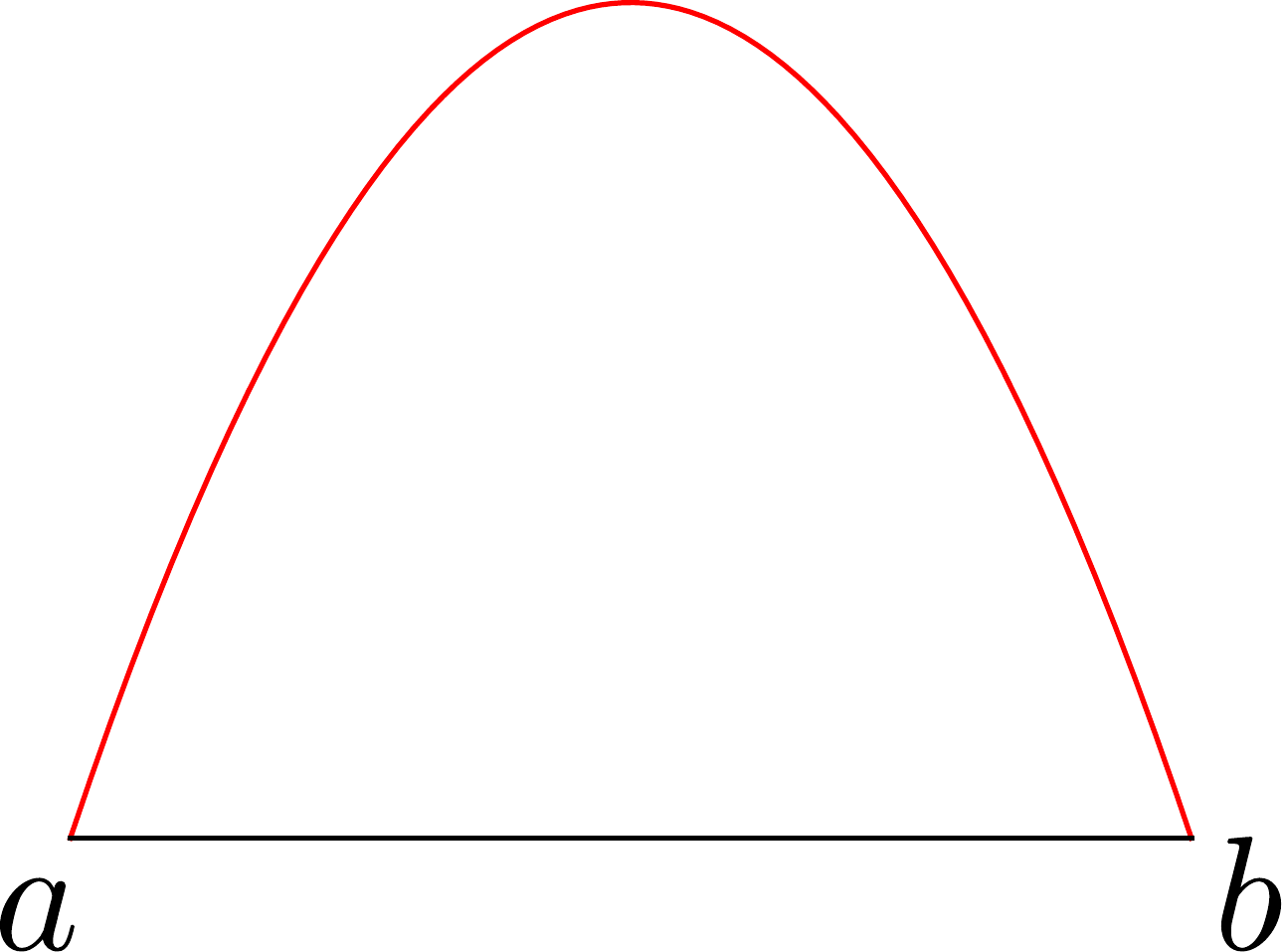}}\quad\quad
\subfigure[]{\includegraphics[width=0.2\textwidth]{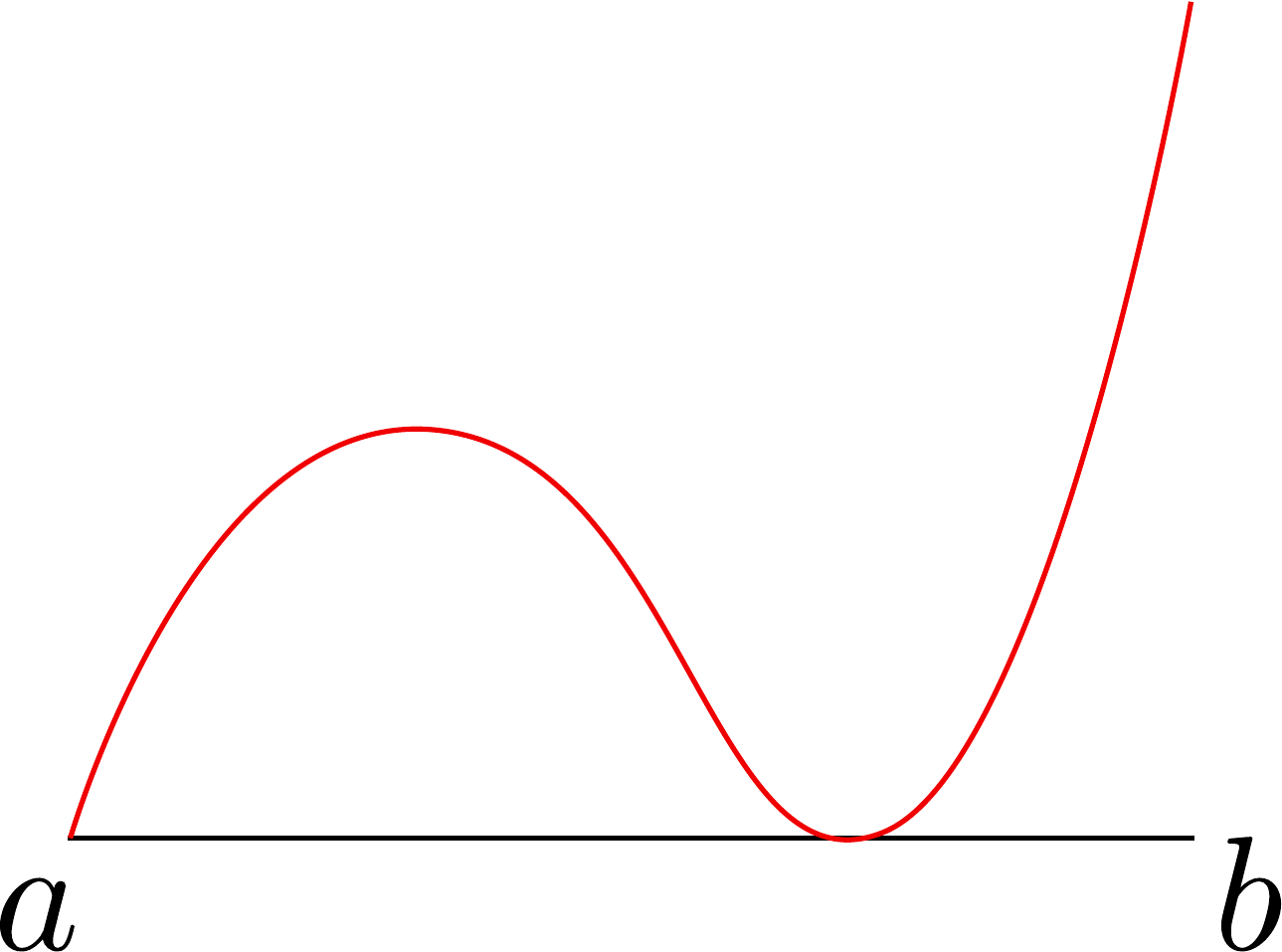}}\quad\quad
\subfigure[]{\includegraphics[width=0.2\textwidth]{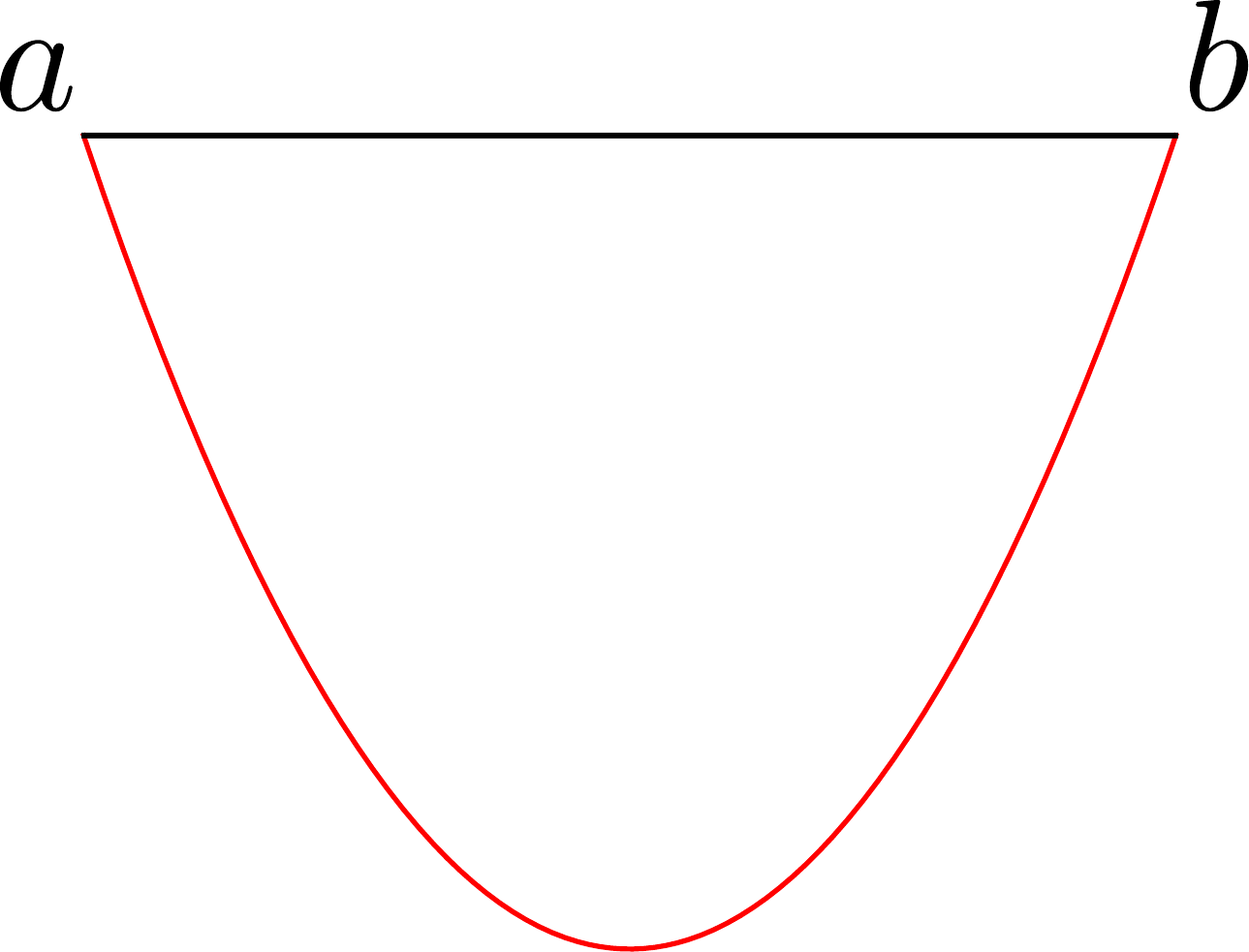}}\quad\quad
\subfigure[]{\includegraphics[width=0.2\textwidth]{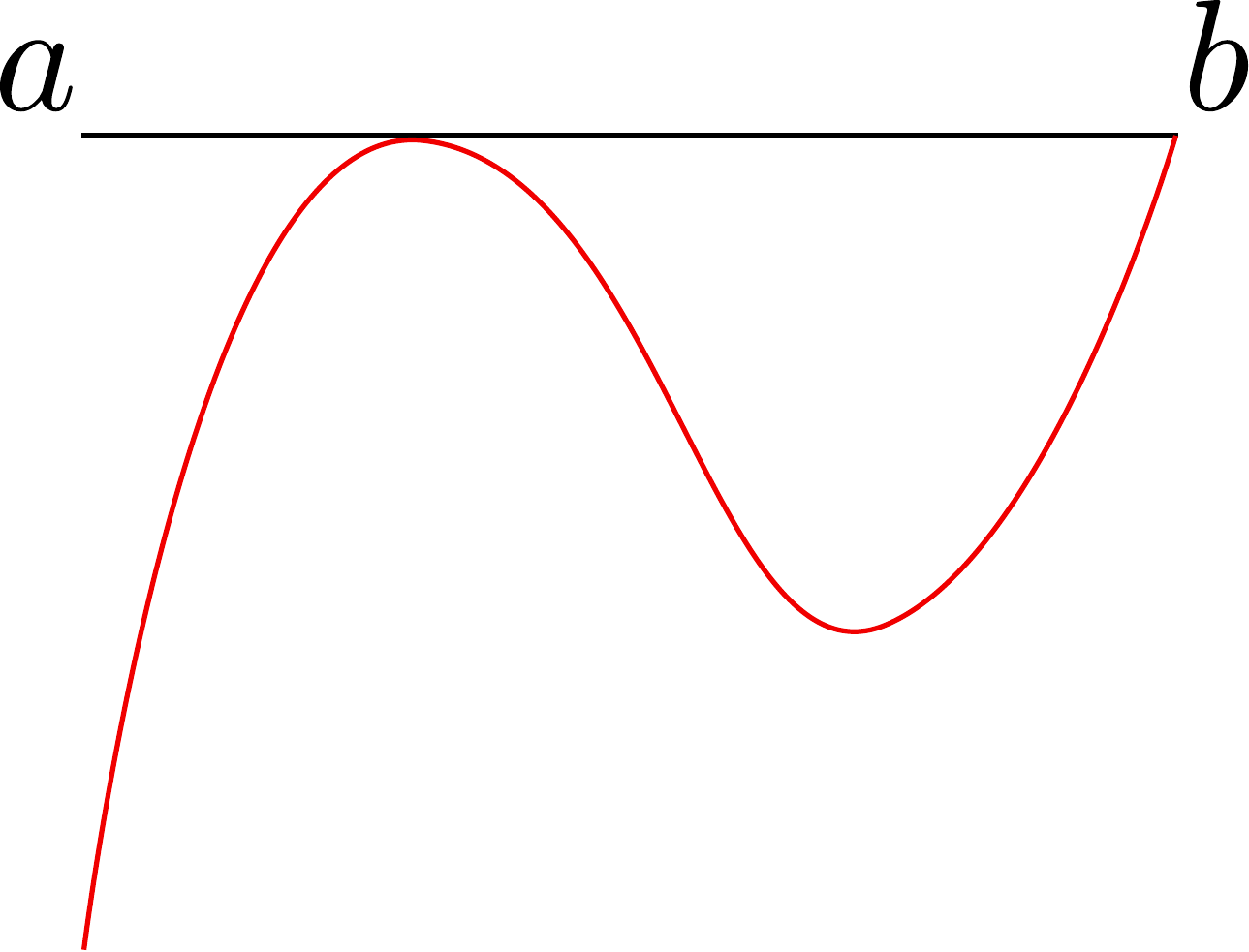}}
\caption{The possible qualitative graphs of $h(x)$, given the constraints provided by theorem \ref{th:analytical}. (a) and (b) are the possibilities for the lower bound case, while (c) and (d) are the possibilities for the upper bound case.}
\label{fig:proof-dsc}
\end{figure}

It follows that the support must be of the form $\{a, c\}$ with $c \in [a,b]$. To actually find the measure we need to impose all constraints, which leads us to the system

\[
\left\{
\begin{array}{l}
p_a + p_c = 1 \\
ap_a + cp_c = \lambda \\
a^2p_a + c^2p_c = \lambda^2 + \sigma^2
\end{array}
\right.
\]
For $\sigma > 0$ there is only one solution:

\[
p_a = \frac{\sigma^2}{\sigma^2 + (\lambda-a)^2}, \quad p_c = \frac{(\lambda-a)^2}{\sigma^2 + (\lambda-a)^2}, \quad c = \lambda +  \frac{\sigma^2}{\lambda-a}.
\]
and evaluating $\expect{g(X)}$ for this measure gives us the lower bound. If we wanted the upper bound instead, the only difference is that now we must have $h(x) \leq 0$, so the possibilities for the qualitative graph of $h(x)$ are the ones in Figs \ref{fig:proof-dsc}(c, d), so the support must be of the form $\{d, b\}$ where $d \in [a,b]$. The rest follows by swapping $a$ for $b$ and $c$ for $d$. \qed
\end{proof}

Let us examine some cases where we can apply theorem \ref{th:dg-strict-convex}

\subsubsection{Moment Generating Functions}
\label{sssec:mom-gen}

Suppose we want to find bounds for the moment generating function $\expect{e^{Xs}}$ of a non-negative random variable $X$. We have then $\mathcal{S} = \mathds{R}_+$ and in order to be able to use theorem \ref{th:dg-strict-convex} we will need to firstly study the case $\mathcal{S} = [0,a]$ and then use theorem \ref{th:lim} to obtain the correct bounds.

If $s > 0$, then $g(x) = e^{xs}$ is such that $g'(x)$ is strictly convex, so $X\in [0,a]$ leads us to the bounds

\[
\frac{\sigma^2 + \lambda^2 e^{\lambda s + \nicefrac{\sigma^2 s}{\lambda}}}{\sigma^2 + \lambda^2} \leq \expect{e^{Xs}} \leq \frac{\sigma^2 e^{as} + (\lambda - a)^2 e^{\lambda s + \nicefrac{\sigma^2 s}{(\lambda - a)}}}{\sigma^2 + (\lambda - a)^2}
\]
whereas if $s < 0$, then $g(x) = -e^{xs}$ is such that $g'(x)$ is strictly convex, so $X\in [0,a]$ implies

\[
\frac{\sigma^2 e^{as} + (\lambda - a)^2 e^{\lambda s + \nicefrac{\sigma^2 s}{(\lambda - a)}}}{\sigma^2 + (\lambda - a)^2} \leq \expect{e^{Xs}} \leq \frac{\sigma^2 + \lambda^2 e^{\lambda s + \nicefrac{\sigma^2 s}{\lambda}}}{\sigma^2 + \lambda^2}
\]
Finally, taking the limit $a\rightarrow \infty$, one arrives at

\[
\expect{e^{Xs}} \geq \frac{\sigma^2 + \lambda^2e^{\nicefrac{(\lambda^2 +\sigma^2)s}{\lambda}}}{\sigma^2 + \lambda^2}
\]
for $s>0$, with no upper bound available and

\[
e^{\lambda s} \leq \expect{e^{Xs}} \leq \frac{\sigma^2 + \lambda^2e^{\nicefrac{(\lambda^2 +\sigma^2)s}{\lambda}}}{\sigma^2 + \lambda^2}
\]
for $s<0$ (note that the lower bound does not improve over Jensen's inequality). These bounds can be visualized more easily graphing them for the cumulant generating function (see figure \ref{fig:gen-func})

\begin{figure}[htb]
\centering
\includegraphics[width=0.45\textwidth]{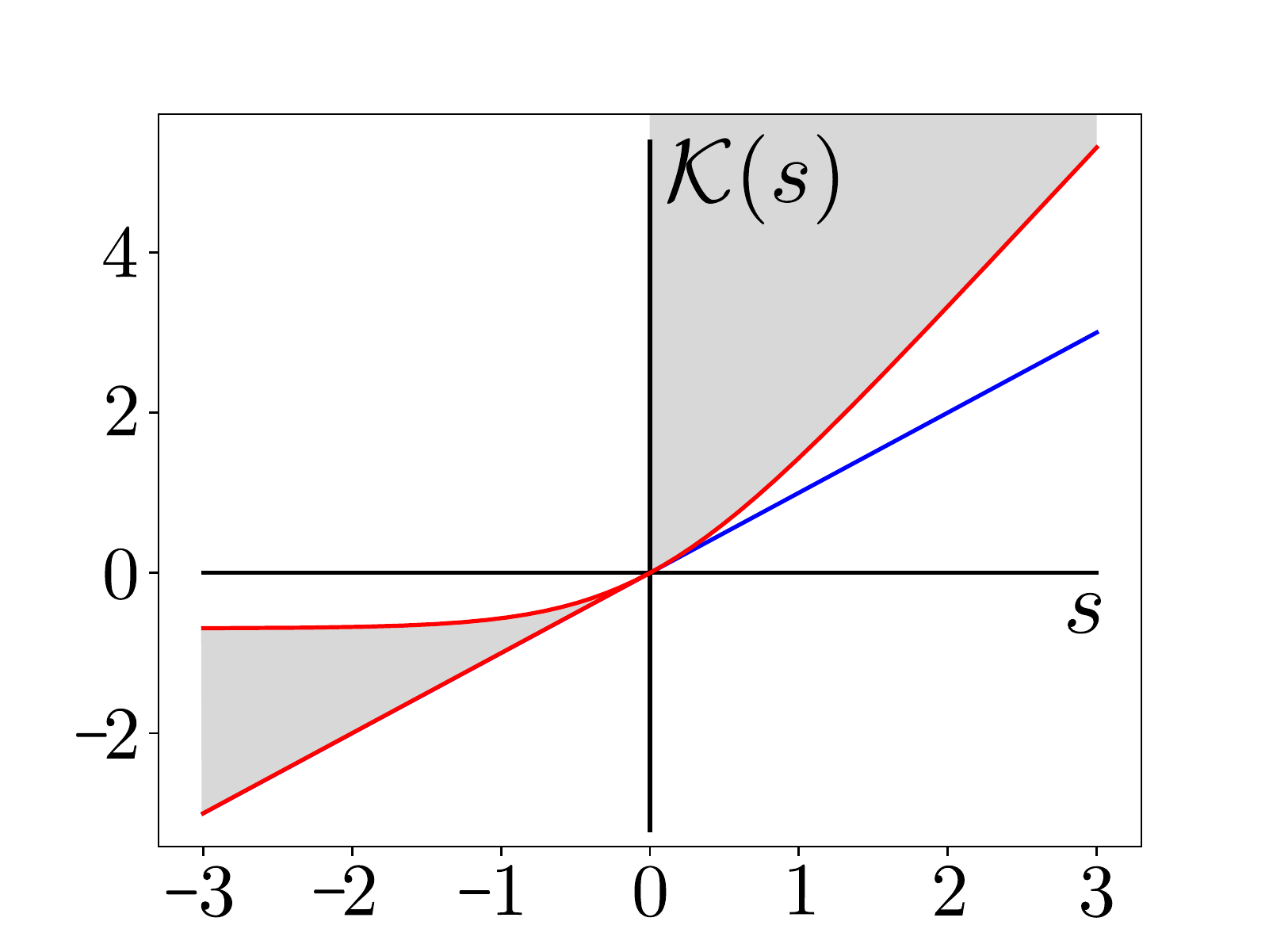}
\caption{Our bounds for the cumulant generating function of a non negative random variable with average and variance equal to 1 (grey region between the red curves) together with the usual Jensen lower bound (blue).}
\label{fig:gen-func}
\end{figure}

\subsubsection{Power Means}
\label{sssec:pow-mean}

Suppose that $X$ is a positive random variable, with average $\lambda$ and variance $\sigma^2$ and that we are interested in finding bounds to the power mean $\expect{X^s}^{\nicefrac{1}{s}}$ for $s\neq 0$. We must study $g(x) = x^s$ with $\mathcal{S} = ]0, \infty[$. Once again, we need to consider a progressive compact cover to use theorem \ref{th:dg-strict-convex} and then apply theorem \ref{th:lim} to obtain the final bound. Since any progressive compact cover of $]0,\infty[$ will do, we can use intervals of the form $[\nicefrac{1}{a}, a]$ and then take the limit $a\rightarrow \infty$.

Without worrying in a first moment which is the lower and which is the upper bound (which will depend on the convexity of the derivatives), the two bounds prescribed in theorem \ref{th:dg-strict-convex} (and their limits for $a\rightarrow \infty$) are as follows.

\noindent When $\nicefrac{1}{a}$ is in the support:
\[
\frac{\frac{\sigma^2}{a^s} + \left(\lambda -\frac{1}{a}\right)^2\left(\lambda + \frac{\sigma^2}{\lambda - \nicefrac{1}{a}}\right)^s}{\sigma^2 + \left(\lambda - \frac{1}{a}\right)^2} \rightarrow 
\left\{
\begin{array}{lr}
\frac{(\sigma^2 + \lambda^2)^{s-1}}{\lambda^{s-2}} & \mbox{if }s>0 \\
\infty & \mbox{if }s<0
\end{array}
\right.
\]
and when $a$ is in the support:

\[
\frac{\sigma^2 a^s + (\lambda - a)^2\left(\lambda - \frac{1}{\lambda - a}\right)^s}{\sigma^2 + (\lambda-a)^2} \rightarrow
\left\{
\begin{array}{lr}
\infty & \mbox{if }s>2 \\
\lambda^s & \mbox{if }s<2
\end{array}
\right.
\]

For $s > 2$ and $0 < s < 1$, $g'(x)$ is strictly convex for $x>0$, whereas for $s<0$ and for $1<s<2$, $-g'(x)$ is strictly convex. If we define $M_s = \expect{X^s}^{\nicefrac{1}{s}}$, this implies that

\[
0 \leq M_s \leq \lambda\quad\mbox{for }s < 0
\]\[
\frac{(\sigma^2 + \lambda^2)^{1-\nicefrac{1}{s}}}{\lambda^{1-\nicefrac{2}{s}}} \leq M_s \leq \lambda\quad\mbox{for }0 < s < 1
\]\[
\lambda \leq M_s \leq \frac{(\sigma^2 + \lambda^2)^{1-\nicefrac{1}{s}}}{\lambda^{1-\nicefrac{2}{s}}}\quad\mbox{for }1 < s < 2
\]\[
\frac{(\sigma^2 + \lambda^2)^{1-\nicefrac{1}{s}}}{\lambda^{1-\nicefrac{2}{s}}} \leq M_s\quad\mbox{for }s > 2\mbox{, with no upper bound available}
\]

These bounds and their comparison with Jensen's inequality can be found in figure \ref{fig:power}

\begin{figure}[htb]
\centering
\subfigure[]{\includegraphics[width=0.4\textwidth]{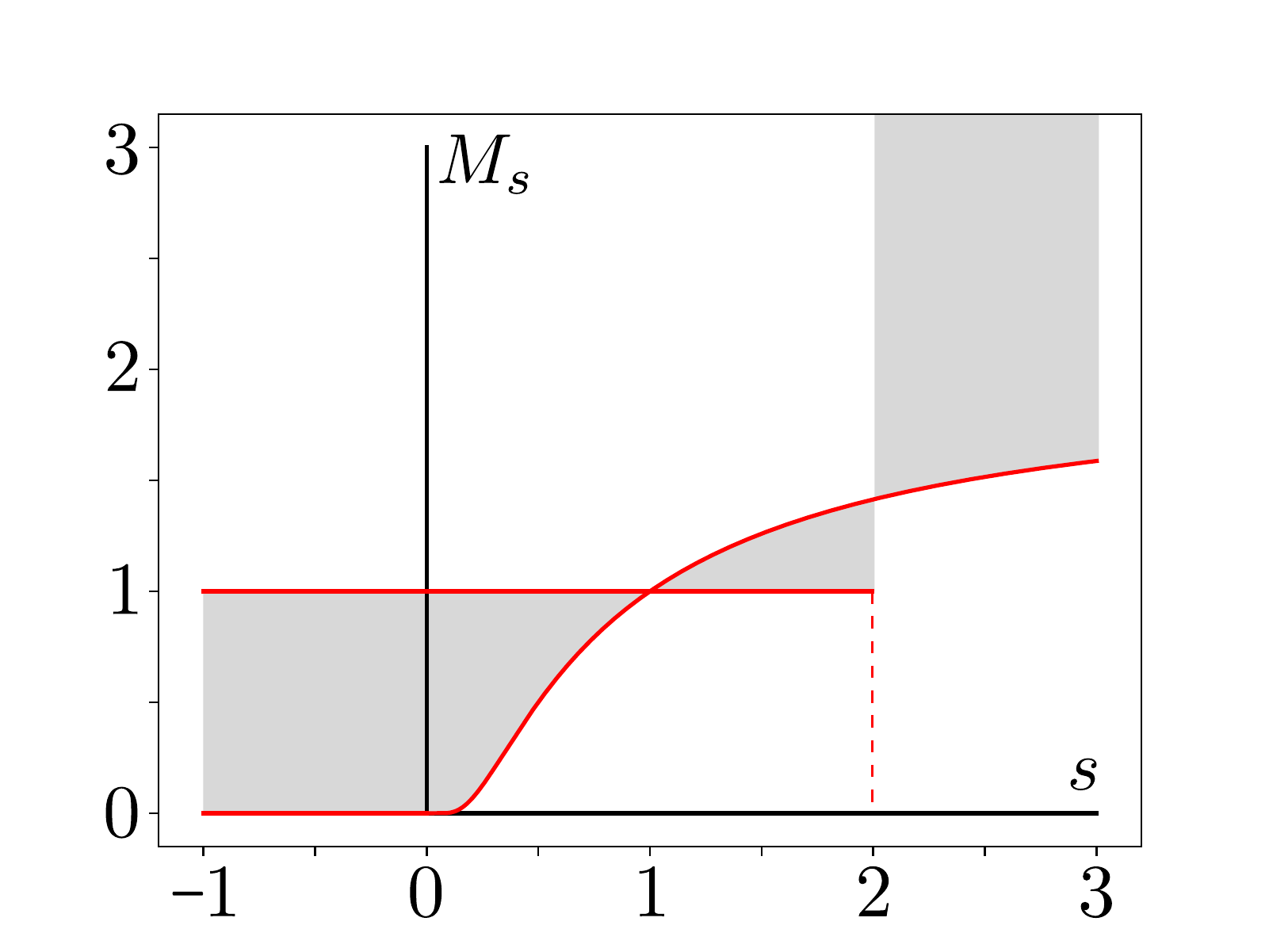}}\quad
\subfigure[]{\includegraphics[width=0.4\textwidth]{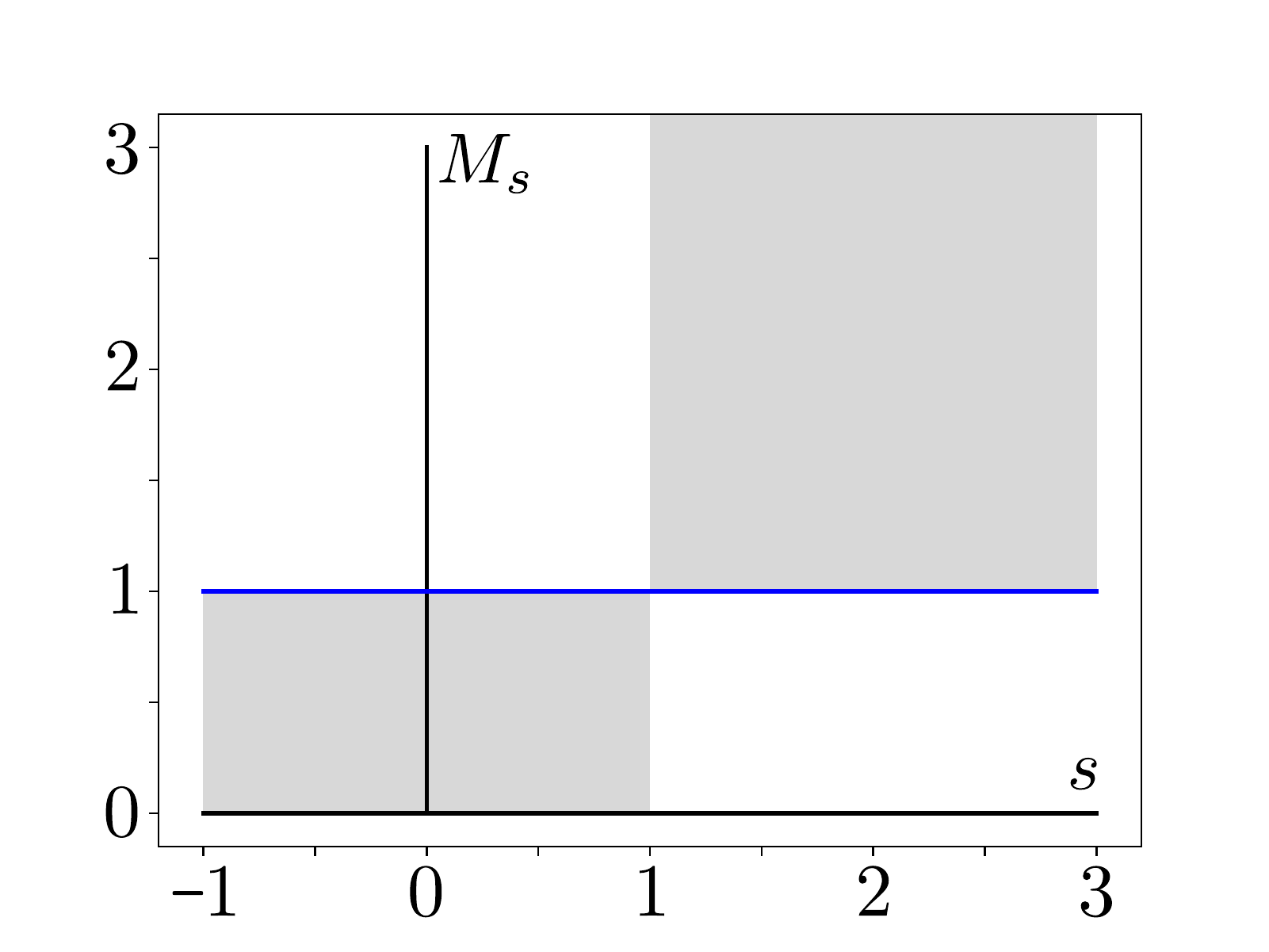}}
\caption{Our bounds for the power mean of a non negative random variable with average and variance equal to 1 (grey region between the red curves) together with the usual Jensen bounds (grey regions and blue curves).}
\label{fig:power}
\end{figure}

\subsection{If $g(x)$ is continuous}
\label{ssec:g-cont}
If we relax the hypothesis and require only that $g$ be continuous, we can still use corollary \ref{cor:Fs} to write the problem of finding the bounds as an extremization over measures with up to 3 points in their support:

\begin{theorem}
Let $g:\mathds{R} \rightarrow \mathds{R}$ continuous, $\lambda \in \mathds{R}$, $\sigma > 0$, $\vec{p} \equiv (p_a, p_b, p_c)$, $S = \{\vec{p} \in \mathds{R}^3\mid p_a, p_b, p_c > 0\,\,\mbox{and}\,\, p_a + p_b + p_c = 1\}$ and define

\[
x_a(\vec{p}, \theta) = \lambda + \sigma\left(\cos(\theta) \sqrt{\frac{p_b}{p_a}} + \sin(\theta)\sqrt{p_c}\right)\frac{1}{\sqrt{p_a + p_b}} \]\[
x_b(\vec{p}, \theta) = \lambda + \sigma\left(-\cos(\theta) \sqrt{\frac{p_a}{p_b}} + \sin(\theta)\sqrt{p_c}\right)\frac{1}{\sqrt{p_a + p_b}}
\]\[
x_c(\vec{p}, \theta) = \lambda - \sigma\sin(\theta)\sqrt{\frac{p_a + p_b}{p_c}}
\]
Then
\[
\inf_{\mu\in M} \expect{g(X)}_{\mu} = \inf_{\vec{p}\in S} \inf_{\theta} \left( p_a g(x_a) + p_b g(x_b) + p_c g(x_c)\right)
\]\[
\sup_{\mu\in M} \expect{g(X)}_{\mu} = \sup_{\vec{p}\in S} \sup_{\theta} \left( p_a g(x_a) + p_b g(x_b) + p_c g(x_c)\right)
\]
where $M$ is the set of measures with support in $\mathds{R}$, $\expect{X}_{\mu} = \lambda$ and $\vari{x}_{\mu} = \sigma^2$.
\label{th:dg-strict-convex2}
\end{theorem}

\begin{proof}
Since $\mathcal{S} = \mathds{R}$ is an $F_{\sigma}$ set and $\Gamma(x) = (x, x^2, g(x))$ is continuous, then we can apply corollary \ref{cor:Fs}. Since $m=2$:

\[
\inf_{\mu\in \MSF} \expect{g(X)}_{\mu} = \inf_{\mu\in \MSM{3}} \expect{g(X)}_{\mu}
\]\[
\sup_{\mu\in \MSF} \expect{g(X)}_{\mu} = \sup_{\mu\in \MSM{3}} \expect{g(X)}_{\mu}
\]
where $\phi = (\lambda, \lambda^2 + \sigma^2)$ (so $\MSF = M$). To characterize the measures in $\MSM{3}$, we must impose the constraints. Calling the points in the support $a,b,c$ we have

\[
\left\{
\begin{array}{l}
p_a + p_b + p_c = 1 \\
ap_a + bp_b + cp_c = \lambda \\
a^2p_a + b^2p_b + c^2p_c = \lambda^2 + \sigma^2
\end{array}
\right.
\]
whose solution is

\[
\left\{
\begin{array}{l}
a = \lambda + \sigma\left(\cos(\theta) \sqrt{\frac{p_b}{p_a}} + \sin(\theta)\sqrt{p_c}\right)\frac{1}{\sqrt{p_a + p_b}} \\
b = \lambda + \sigma\left(-\cos(\theta) \sqrt{\frac{p_a}{p_b}} + \sin(\theta)\sqrt{p_c}\right)\frac{1}{\sqrt{p_a + p_b}} \\
c = \lambda - \sigma\sin(\theta)\sqrt{\frac{p_a + p_b}{p_c}}
\end{array}
\right.
\]
where the probabilities are constrained by $(p_a, p_b, p_c) \in S$ and $\theta$ is a free variable. From here the theorem follows from extremizing over these measures (note that the case with exactly 2 points in the support can be ignored, as we can always make $p_c \rightarrow 0$ in a way that $c$ does not contribute to $\expect{g(X)}$ by making $\theta \rightarrow 0$ in a convenient way). \qed
\end{proof}

This theorem also illustrates how to use these results when the constraints are not enough to reduce the possibilities to a single measure. We are left with an optimization problem over the measures satisfying the constraints.

\section{An alternative approach more suited for numerics}
\label{sec:numerics}

With the exception of section \ref{ssec:g-cont}, the cases we analysed so far could be tackled analytically. This was mostly because the number $n$ of random variables in the vector $X$ and the number $m$ of constraints was small, together with other properties that allowed us to reduce the size of the support. As an illustration, in the case where we can only apply corollary \ref{cor:Fs}, the measures that extremize $\expect{g(X)}_{\mu}$ can have in their support up to $m+1$ points in $\mathcal{S}\subseteq\mathds{R}^n$, which corresponds to $(m+1)(n+1)$ variables (for each unknown point in the support, each of the $n$ coordinates and the probability of that point are variables to be found), whereas we have only $m+1$ constraints (the $m$ constraints given by $\phi$ plus normalization of the measure). So we are still left with an optimization problem over the $(m+1)n$ remaining variables, which in general will be a nonlinear program (as seen in section \ref{ssec:g-cont}). The complexity of finding the bounds would then scale exponentially with $(m+1)n$ and quickly become numerically unfeasible.

This situation can be somewhat remedied if we look at what we have been doing from a different angle. The functions 
\[
h_{\pm}(x) = \innp{\alpha_{\pm}}{f(x)} + \beta_{\pm} g(x) + c_{\pm}
\]
that appear in theorem \ref{th:analytical} can be thought as establishing inequalities $\expect{h_-(X)}_{\mu}\geq 0$ and $\expect{h_+(X)}_{\mu}\leq 0$. Substituting the constraints, these give bounds to $\expect{g(X)}_{\mu}$ that can be then optimized by changing the $\alpha_{\pm}$, $\beta_{\pm}$ and $c_{\pm}$, until a measure satisfying $\expect{h(X)}_{\mu} = 0$ (or a sequence of measures, whose expectations converge to 0) can be found. The main result is summarized in the following theorem

\begin{theorem}
Let $\mathcal{S}\subseteq\mathds{R}^n$. If $\phi\notin\bound{\mathcal{D}(\mathcal{S})
}$, then
\[
\inf_{\mu\in\MSF} \expect{g(X)}_{\mu} = \sup_{\alpha\in\mathds{R}^m}\left(\inf_{x\in\mathcal{S}} \left(g(x) + \innp{\alpha}{(f(x) - \phi)}\right)\right)
\]\[
\mbox{and}\quad\quad\sup_{\mu\in\MSF} \expect{g(X)}_{\mu} = \inf_{\alpha\in\mathds{R}^m}\left(\sup_{x\in\mathcal{S}} \left(g(x) + \innp{\alpha}{(f(x) - \phi)}\right)\right)
\]
\label{th:numerical}
\end{theorem}

Before proceeding with the proof we will need the following lemmas

\begin{lemma}
$\phi\notin\clo{\mathcal{D}(\mathcal{S})}$ if and only if there exists $\alpha\in\mathds{R}^m$ and $\varepsilon > 0$ such that $\innp{\alpha}{(f(x)-\phi)} > \varepsilon\,\,\forall\,\, x\in\mathcal{S}$
\label{lem:feasible}
\end{lemma}

\begin{proof}
If $\phi\notin\clo{\mathcal{D}(\mathcal{S})}$, then since $\clo{\mathcal{D}(\mathcal{S})}$ and $\{\phi\}$ are convex and closed and $\{\phi\}$ is also compact, then by the Separating Hyperplane Theorem, there exists a hyperplane $\Pi$ that separates both sets with a gap, that is, if $\Pi$ is defined by $\innp{\alpha}{x} + \beta = 0$, then there exists $\varepsilon > 0$ such that $\innp{\alpha}{\phi} + \beta < \nicefrac{-\varepsilon}{2}$ and $\innp{\alpha}{y} + \beta > \nicefrac{\varepsilon}{2}\,\,\forall\,\, y\in\clo{\mathcal{D}(\mathcal{S})}$. Combining both inequalities it follows that $\innp{\alpha}{(y-\phi)} > \varepsilon\,\,\forall\,\, y\in\clo{\mathcal{D}(\mathcal{S})}$. Since $\mathcal{D}(\mathcal{S}) = \hull{f(\mathcal{S})}$, then $f(\mathcal{S}) \subseteq \clo{\mathcal{D}(\mathcal{S})}$ and hence $\innp{\alpha}{(f(x)-\phi)} > \varepsilon\,\,\forall\,\, x\in\mathcal{S}$.

On the other hand, if $\innp{\alpha}{(f(x)-\phi)} > \varepsilon\,\,\forall\,\, x\in\mathcal{S}$, then taking an expectation on both sides, we have that $\innp{\alpha}{(\expect{f(x)}_{\mu}-\phi)} > \varepsilon\,\,\forall\,\, \mu\in\MS$. Using corollary \ref{co:existence}, this means that $\innp{\alpha}{(y-\phi)} > \varepsilon\,\,\forall\,\, y\in\mathcal{D}(\mathcal{S})$ and this implies that $\innp{\alpha}{(y-\phi)} \geq \varepsilon\,\,\forall\,\, y\in\clo{\mathcal{D}(\mathcal{S})}$. Since taking $y = \phi$ implies that $\innp{\alpha}{(y-\phi)} = 0$ and $\varepsilon > 0$, this finally implies $\phi\notin\clo{\mathcal{D}(\mathcal{S})}$. \qed
\end{proof}

\begin{lemma}
If $\phi\notin\bound{\mathcal{D(\mathcal{S})}}$, then there exists a progressive cover $(\mathcal{S}_k)_{k=1}^{\infty}$ of $\mathcal{S}$, such that for all $k$, $\Gamma(\mathcal{S}_k)$ is bounded and
\begin{itemize}
\item $\phi\notin\clo{\mathcal{D}(\mathcal{S}_k)}$ if $\phi\notin\clo{\mathcal{D(\mathcal{S})}}$
\item $\phi\in\inter{\mathcal{D}(\mathcal{S}_k)}$ if $\phi\in\inter{\mathcal{D(\mathcal{S})}}$
\end{itemize}
\label{lemma:cover}
\end{lemma}

\begin{proof}
Let $(R_k)_{k=1}^{\infty}$ be a progressive bounded cover of $\mathds{R}^{m+1}$ (like $[-k,k]^{m+1}$). We will use it to build a progressive bounded cover $(\mathcal{B}_k)_{k=1}^{\infty}$ of $\Gamma(\mathcal{S})$, by taking $\mathcal{B}_k = \Gamma(\mathcal{S})\cap R_k$, which we'll in turn use to find a progressive cover $(\mathcal{S}_k)_{k=1}^{\infty}$ of $\mathcal{S}$, by taking $\mathcal{S}_k = \mathcal{S} \cap \Gamma^{-1}(\mathcal{B}_k)$. Since $\Gamma(\mathcal{S}_k) = \mathcal{B}_k$ by construction, then all that is left is to show that we can find a cover $(R_k)$ such that the relation between $\phi$ and $\mathcal{D}(\mathcal{S}_k)$ is as predicted by the theorem.

It suffices to show that the relation between $\phi$ and $\mathcal{D}(\mathcal{S}_k)$ is as prescribed for all $k>N_{\phi}$ (as we could just remove the first elements of the cover), which must actually happen regardless of what $(R_k)$ is. To see this, note first that $(\mathcal{S}_k)$ being a progressive cover of $\mathcal{S}$ implies $\mathcal{D}(\mathcal{S}_k) \subseteq \mathcal{D}(\mathcal{S})$, which is enough to prove the case where $\phi\notin\clo{\mathcal{D}(\mathcal{S})}$. For the case $\phi\in\inter{\mathcal{D}(\mathcal{S})}$, since this means that $\phi\in\inter{\hull{f(\mathcal{S})}}$, then $\phi\in\inter{S}$, where $S$ is a simplex such that $S \subseteq \inter{\hull{f(\mathcal{S})}}$. Applying Caratheodory's theorem on the vertexes $V_i$ of $S$, it follows that each $V_i$ is a convex combination of points in a finite subset of $f(\mathcal{S})$ and hence if $F_{\phi}$ is the union of these subsets, we have $\phi\in\inter{\hull{F_{\phi}}}$. Since $F_{\phi}\subseteq f(\mathcal{S})$ is finite and $(\mathcal{S}_k)$ is a progressive cover of $\mathcal{S}$, then for a sufficiently large $k$ ($k>N_{\phi}$) we must have $F_{\phi} \subseteq f(\mathcal{S}_k)$ and hence $\phi\in\inter{\hull{f(\mathcal{S}_k)}} = \inter{\mathcal{D}(\mathcal{S}_k)}$, completing the proof. \qed
\end{proof}

\begin{proof}[Theorem \ref{th:numerical}]
We'll do the proof only for $\inf\left(\expect{g(X)}\right)$, as the proof for the supremum follows from considering the infimum for $-g$. Defining

\[
\sigma = \inf_{\mu\in\MSF} \expect{g(X)}_{\mu}
\]
we will start with the case where $\sigma$ is finite. In this case, the hypothesis for theorem \ref{th:analytical} are satisfied, so consider the values $\alpha_-^i$, $\beta_-$ and $c_-$ predicted by it (that will be denoted $\alpha_i$, $\beta$ and $c$ for simplicity). As seen in the proof of theorem \ref{th:analytical}, $a = (\alpha_1,\ldots, \alpha_m, \beta)$ is a vector normal to a supporting hyperplane $\Pi$ of $\hull{\Gamma(\mathcal{S})}$ that passes through $(\phi, \sigma) \in \clo{\hull{\Gamma(\mathcal{S})}}$.

Suppose that we had $\beta = 0$. In this case, if we project all points in $\Pi$ into the hyperplane $z = 0$, we get $\Pi'\times \{0\}$ (instead of $\mathds{R}^m\times \{0\}$), where $\Pi'$ is a hyperplane in $\mathds{R}^m$ with normal vector $(\alpha_1,\ldots, \alpha_m)$. Furthermore, if we project all points in $\hull{\Gamma(\mathcal{S})}$ into the $z=0$ hyperplane, we get $\mathcal{D}(\mathcal{S}) \times \{0\}$ (because of corollary \ref{co:existence}). But since $\Pi$ is a supporting hyperplane of $\hull{\Gamma(\mathcal{S})}$, then $\Pi'$ will be a supporting hyperplane of $\mathcal{D}(\mathcal{S})$ passing through $\phi$. As a consequence this would imply that $\phi \in \bound{\mathcal{D}(\mathcal{S})}$. Our point is that the hypothesis that $\phi\notin\bound{\mathcal{D}(\mathcal{S})}$ implies then that $\beta > 0$ and as such, without loss of generality, we can choose $\beta = 1$ for the values predicted by theorem \ref{th:analytical}, that is
\begin{equation}
\innp{\alpha}{f(x)} + g(x) + c \geq 0\,\,\,\forall\, x\in \mathcal{S}
\label{eq:g-ineq}
\end{equation}

\begin{equation}
\inf_{\mu\in\MSF} \expect{\innp{\alpha}{f(X)} + g(X) + c}_{\mu} = 0 \Rightarrow 
\sigma = - \innp{\alpha}{\phi} - c.
\label{eq:lolol}
\end{equation}
Consider now the set

\[
\mathcal{C}=\left\{(\alpha, z)\in\mathds{R}^{m+1}\,\middle|\,g(x) + \innp{\alpha}{(f(x) - \phi)} + z \geq 0\,\,\,\forall\, x\in\mathcal{S}\right\}
\]
The definition of $\mathcal{C}$ implies that if $\mu\in\MSF$ and $(\alpha, z) \in \mathcal{C}$, then $\expect{g(X)}_{\mu} \geq -z$.
However this implies that

\begin{equation}
\sigma = -\inf_{(\alpha, z) \in \mathcal{C}} z
\label{eq:lol-b}
\end{equation}
To see why this is true, we first note that $(\alpha, -\sigma) \in \mathcal{C}$ (which follows directly from substituting $\sigma = - \innp{\alpha}{\phi} - c$ into the definition of $\mathcal{C}$, while using (\ref{eq:g-ineq})) and that $(\alpha,z)\in \mathcal{C} \Rightarrow (\alpha,z')\in \mathcal{C}\,\,\forall\,\, z' > z$. So if we suppose by absurd that $\inf_{(\alpha, z) \in \mathcal{C}} (z) \neq -\sigma$, then there would need to exist $(\alpha, z')\in \mathcal{C}$ such that $z' < -\sigma$. But then we would have

\[
g(x) + \innp{\alpha}{(f(x) - \phi)} \geq -z' \,\,\,\forall\, x\in\mathcal{S} \Rightarrow
\expect{g(X)}_{\mu} \geq - z' >  \sigma \,\,\forall\,\, \mu\in\MSF
\]
which contradicts the definition of $\sigma$.

It also follows from its definition that $\mathcal{C}$ is convex and the epigraph of some convex function $F(\alpha)$. The definition of $\mathcal{C}$ leads us easily to
\[
F(\alpha) = -\inf_{x\in\mathcal{S}}\left(g(x) + \innp{\alpha}{(f(x) - \phi)}\right)
\]
Hence, if we calculate the infimum in equation (\ref{eq:lol-b}) by first taking the infimum over $z$ and then over $\alpha$ we get

\[
\sigma = - \inf_{(\alpha,z)\in \mathcal{C}} z = -\inf_{\alpha\in\mathds{R}^m}(F(\alpha)) = \sup_{\alpha\in\mathds{R}^m}(-F(\alpha)) \Rightarrow
\]\[
\inf_{\mu\in\MSF} \expect{g(X)}_{\mu} = \sup_{\alpha\in\mathds{R}^m}\left(\inf_{x\in\mathcal{S}} \left(g(x) + \innp{\alpha}{(f(x) - \phi)}\right)\right)
\]

This concludes the proof of the case where $\sigma$ is finite. For the case where $\sigma$ is infinite, we start with the progressive cover $(\mathcal{S}_k)_{k=1}^{\infty}$ predicted by lemma \ref{lemma:cover}. For the case $\sigma = \infty$, since this implies $\phi\notin\clo{\mathcal{D}(\mathcal{S})}$, then $\phi\notin\clo{\mathcal{D}(\mathcal{S}_k)}$. Using lemma \ref{lem:feasible}, there exists $\alpha\in\mathds{R}^m$ and $\varepsilon > 0$ such that $\innp{\alpha}{(f(x)-\phi)} > \varepsilon\,\,\forall\,\, x\in\mathcal{S}_k$. It follows that for all $\lambda > 0$ we have
\[
\inf_{x\in\mathcal{S}_k} \left(g(x) + \lambda \innp{\alpha}{(f(x) - \phi)}\right) \geq
\inf_{x\in\mathcal{S}_k} g(x) + \inf_{x\in\mathcal{S}_k} \left(\lambda \innp{\alpha}{(f(x) - \phi)}\right) \geq
\lambda \varepsilon + \inf_{x\in\mathcal{S}_k} g(x)
\]
and since $g(\mathcal{S}_k)$ is bounded (because of lemma \ref{lemma:cover}), then taking $\lambda\rightarrow\infty$ it follows that
\[
\inf_{\mu\in M(\mathcal{S}_k, \phi)}\expect{g(X)}_{\mu} = \sup_{\alpha\in\mathds{R}^m}\left(\inf_{x\in\mathcal{S}_k} \left(g(x) + \innp{\alpha}{(f(x) - \phi)}\right)\right) = \infty
\]
and applying theorem \ref{th:lim} completes the proof of this case.

Finally, for the case $\sigma = -\infty$ it follows that $\phi\in\inter{\mathcal{D}(\mathcal{S}_k)}$ and since $g(\mathcal{S}_k)$ is bounded, then $\inf \left(\expect{g(X)}_{\mu}\right)$ for $\mu\in M(\mathcal{S}_k,\phi)$ is finite. Using what we already proved for the case where $\sigma$ is finite, it follows that

\[
\inf_{\mu\in M(\mathcal{S}_k,\phi)} \expect{g(X)}_{\mu} = \sup_{\alpha\in\mathds{R}^m}\left(\inf_{x\in\mathcal{S}_k} \left(g(x) + \innp{\alpha}{(f(x) - \phi)}\right)\right)
\]
If we combine this result with theorem \ref{th:lim} we get

\[
-\infty = \inf_{\mu\in M(\mathcal{S},\phi)} \expect{g(X)}_{\mu} = \lim_{k\rightarrow\infty} \left( \inf_{\mu\in M(\mathcal{S}_k,\phi)} \expect{g(X)}_{\mu} \right)=
\lim_{k\rightarrow\infty} \left( \sup_{\alpha\in\mathds{R}^m}\left(\inf_{x\in\mathcal{S}_k} \left(g(x) + \innp{\alpha}{(f(x) - \phi)}\right)\right)\right)
\]
So for every $\delta \in \mathds{R}$ there exists $N_{\delta}$ such that if $k > N_{\delta}$ then

\[
\sup_{\alpha\in\mathds{R}^m}\left(\inf_{x\in\mathcal{S}_k} \left(g(x) + \innp{\alpha}{(f(x) - \phi)}\right)\right) < \delta \Rightarrow
\]\[
\inf_{x\in\mathcal{S}_k} \left(g(x) + \innp{\alpha}{(f(x) - \phi)}\right) < \delta \,\,\forall\,\,\alpha\in\mathds{R}^m
\]\[
\mbox{But}\quad\inf_{x\in\mathcal{S}} \left(g(x) + \innp{\alpha}{(f(x) - \phi)}\right) \leq \inf_{x\in\mathcal{S}_k} \left(g(x) + \innp{\alpha}{(f(x) - \phi)}\right) < \delta \,\,\forall\,\,\alpha\in\mathds{R}^m \Rightarrow
\]\[
\inf_{x\in\mathcal{S}} \left(g(x) + \innp{\alpha}{(f(x) - \phi)}\right) < \delta \,\,\forall\,\,\alpha\in\mathds{R}^m, \delta\in\mathds{R} \Rightarrow
\]\[
\inf_{x\in\mathcal{S}} \left(g(x) + \innp{\alpha}{(f(x) - \phi)}\right) = -\infty \,\,\forall\,\,\alpha\in\mathds{R}^m \Rightarrow
\]\[
\sup_{\alpha\in\mathds{R}^m}\left(\inf_{x\in\mathcal{S}} \left(g(x) + \innp{\alpha}{(f(x) - \phi)}\right)\right) = -\infty = \inf_{\mu\in\MSF} \expect{g(X)}_{\mu}
\]
completing the proof of this case. \qed
\end{proof}

The main advantage of this formulation is that finding the bounds becomes a convex optimization problem in $\mathds{R}^m$. In particular, we minimize some convex function $F(\alpha)$, where evaluating $F$ is akin to solving a global optimization in $\mathcal{S}\subseteq\mathds{R}^n$. More precisely, if we define

\[
G(x;\alpha) = g(x) + \innp{\alpha}{(f(x) - \phi)}
\]
then the convex functions we must use for finding the lower and upper bounds of $\expect{g(X)}_{\mu}$ are

\[
F_{\pm}: \mathds{R}^m\rightarrow \overline{\mathds{R}}\quad\mbox{such that}\quad F_{\pm}(\alpha) = \sup_{x\in\mathcal{S}} \left(\pm G(x;\alpha)\right),
\]
as we have

\begin{equation}
\inf_{\mu\in\MSF} \expect{g(X)}_{\mu} = - \inf_{\alpha\in\mathds{R}^m} F_-(\alpha)\quad\quad\mbox{and}\quad\quad \sup_{\mu\in\MSF} \expect{g(X)}_{\mu} = \inf_{\alpha\in\mathds{R}^m} F_+(\alpha)
\label{eq:bounds-F1}
\end{equation}

If we consider the situation where $\phi\in\inter{\mathcal{D}(\mathcal{S})}$ (that is the constraints on the expectations are feasible and the hypothesis of theorem \ref{th:numerical} is obeyed), then the complexity of solving the problem numerically with this approach grows polynomially in $m$ and exponentially in $n$, which is a huge improvement over the more naive approach of the previous sections. Nevertheless, this approach is not as useful for obtaining analytical results and makes it harder to use special properties of the $f_i$ and $g$, so there is actually a tradeoff between the two approaches. Finally, the intermediate steps of the minimization of both $F_{\pm}$ can be used to create bounds that are looser but numerically cheaper to obtain (only a rough idea of where the extrema are might already lead to an useful bound):

\begin{corollary}
If $\phi\in\inter{\mathcal{D}(\mathcal{S})}$
\[
- F_-(\alpha) \leq \expect{g(X)}_{\mu} \leq F_+(\alpha') \quad\quad\forall\,\, \alpha, \alpha' \in\mathds{R}^m\,\,\mbox{and}\,\,\mu\in\MSF
\]
\label{cor:loose-bounds}
\end{corollary}
\begin{proof}
This follows directly from equation (\ref{eq:bounds-F1}) \qed
\end{proof}

\subsection{A subgradient for $F_{\pm}$}

In order to find the value of $F_{\pm}(\alpha)$ for a given $\alpha$, we will need to solve an optimization problem in $\mathcal{S}$. Doing this numerically will typically lead us to sequences $(x^{\pm}_n)_{n=1}^{\infty}$ in $\mathcal{S}$, such that

\[
\lim_{n\rightarrow\infty} G(x^{\pm}_n;\alpha) = \pm F_{\pm}(\alpha).
\]
Interestingly, if the sequences $(f(x^{\pm}_n))_{n=1}^{\infty}$ are convergent, they can be used to find a subgradient for $F_{\pm}(\alpha)$, without the need for evaluating $F_{\pm}$ for different values of $\alpha$. More precisely

\begin{lemma}
If $\alpha$ is such that $F_{\pm}(\alpha)$ is finite, $(x^{\pm}_n)_{n=1}^{\infty}$ is such that

\[
\lim_{n\rightarrow\infty} G(x^{\pm}_n;\alpha) = \pm F_{\pm}(\alpha). \quad\mbox{ and the limit }\quad
S_{\pm}(\alpha) = \pm \lim_{n\rightarrow\infty} (f(x_n^{\pm}) - \phi)
\]
is convergent, then $S_{\pm}(\alpha)$ is a subgradient of $F_{\pm}$ at $\alpha$.
\end{lemma}

\begin{proof}
The proof is by direct verification. We need to show that for all $\alpha '\in\mathds{R}^m$ we have

\[
F_{\pm}(\alpha ') - F_{\pm}(\alpha) - \innp{S_{\pm}(\alpha)}{(\alpha ' - \alpha)} \geq 0
\]
If $F_{\pm}(\alpha ')$ is not finite, then the left hand side is $\infty$ (its definition implies that $F_{\pm}(\alpha ')$ cannot be $-\infty$) and the inequality follows trivially, so we only need to consider the cases where $F_{\pm}(\alpha ')$ is finite.
For $F_+$:

\[
F_{+}(\alpha ') - F_{+}(\alpha) - \innp{S_{+}(\alpha)}{(\alpha ' - \alpha)} = 
\]\[
\sup_{x\in\mathcal{S}} G(x;\alpha ') - \lim_{n\rightarrow\infty} \left(g(x_n^+) + \cancel{\innp{\alpha}{(f(x_n^+) - \phi})}\right) - \lim_{n\rightarrow\infty} \innp{(f(x_n^+) - \phi)}{(\alpha ' - \cancel{\alpha})} = 
\]\[
\sup_{x\in\mathcal{S}} G(x;\alpha ') - \lim_{n\rightarrow\infty} \left(g(x_n^+) + \innp{\alpha '}{(f(x_n^+) - \phi})\right) = \sup_{x\in\mathcal{S}} G(x;\alpha ') - \lim_{n\rightarrow\infty} G(x_n^+;\alpha ') \geq 0
\]

For $F_-$:

\[
F_{-}(\alpha ') - F_{-}(\alpha) - \innp{S_{-}(\alpha)}{(\alpha ' - \alpha)} = 
\]\[
\sup_{x\in\mathcal{S}} (-G(x;\alpha ')) + \lim_{n\rightarrow\infty} \left(g(x_n^+) + \cancel{\innp{\alpha}{(f(x_n^+) - \phi})}\right) + \lim_{n\rightarrow\infty} \innp{(f(x_n^+) - \phi)}{(\alpha ' - \cancel{\alpha})} = 
\]\[
- \inf_{x\in\mathcal{S}} G(x;\alpha ') + \lim_{n\rightarrow\infty} \left(g(x_n^+) + \innp{\alpha '}{(f(x_n^+) - \phi})\right) = \lim_{n\rightarrow\infty} G(x_n^+;\alpha ') - \inf_{x\in\mathcal{S}} G(x;\alpha ') \geq 0
\] \qed
\end{proof}

This result implies that a subgradient method can be used to obtain the bounds numerically, under no extra assumptions about $f$ and $g$. Also, note that if $\mathcal{S}$ is compact and $f,g$ are continuous in $\mathcal{S}$, then no limits need to be taken and we can just use the estimates for $\argmax_{x\in\mathcal{S}}(\pm G(x;\alpha))$ obtained when calculating $F_{\pm}(\alpha)$.

\subsection{An analytical example}
\label{ssec:cavina}

Let $X$ be a random variable such that $X\geq a$, $a<0$ and $\expect{e^X} = 1$. Given some value $\lambda > a$ we are interested in the largest value possible for $\prob{X \geq \lambda}$. This problem was studied in \cite{Cavina-2016} (in the context of finding the optimal work extraction of a process obeying Jarzynski's equality \cite{Jarzynski-1997}) where it was found that

\begin{equation}
\prob{X \geq \lambda} \leq \min\left\{1, \frac{1-e^a}{e^{\lambda} - e^a}\right\}
\label{eq:cavina}
\end{equation}
holds and is sharp.

We can obtain the same result with theorem \ref{th:numerical}. We have in this case $\mathcal{S} = [a,\infty[$, $g(x) = \Theta(x-\lambda)$, $f(x)=e^x$, $\phi=1$ and $\mathcal{D}(\mathcal{S}) = [e^a, \infty[$. Since $\phi\in\inter{\mathcal{D}(\mathcal{S})}$, then theorem \ref{th:numerical} tells us that the answer is

\[
\sigma_+ \equiv \inf_{\alpha}\left(\sup_{x \geq a} \left(\Theta(x-\lambda) + \alpha (e^x - 1)\right)\right)
\]
One can easily determine that

\[
F(\alpha)\equiv \sup_{x \geq a} \left(\Theta(x-\lambda) + \alpha (e^x - 1)\right) = 
\left\{
\begin{array}{ll}
\infty, & \mbox{if }\alpha > 0 \\
\alpha(e^{\lambda} - 1) + 1, & \mbox{if } \frac{1}{e^a - e^{\lambda}}\leq \alpha\leq 0\\
\alpha(e^a - 1), & \mbox{if }\alpha \leq \frac{1}{e^a - e^{\lambda}}
\end{array}
\right.
\]

\begin{figure}[hbtp!]
\centering
\subfigure[]{\includegraphics[width=0.3\textwidth]{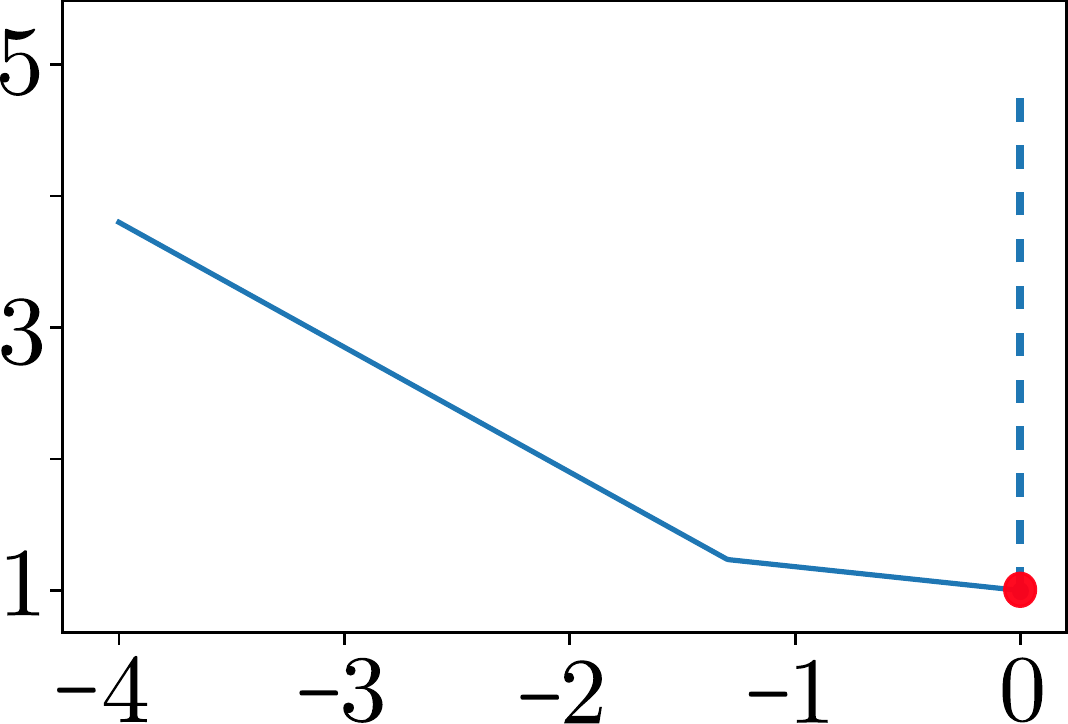}}\quad\quad
\subfigure[]{\includegraphics[width=0.3\textwidth]{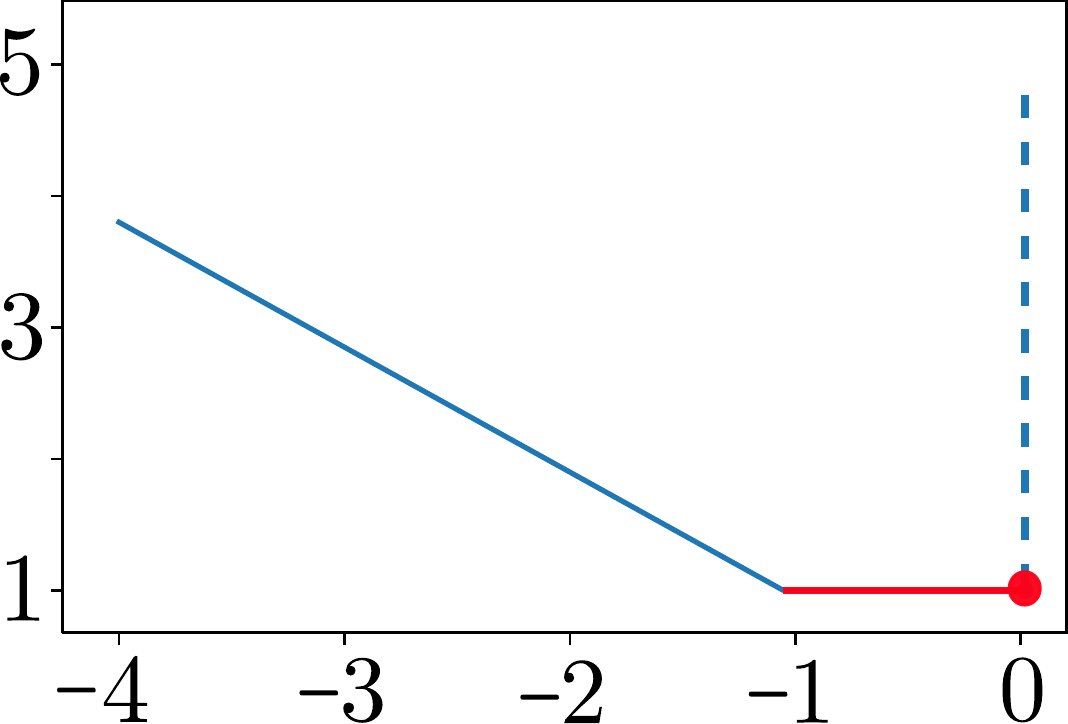}}\quad\quad
\subfigure[]{\includegraphics[width=0.3\textwidth]{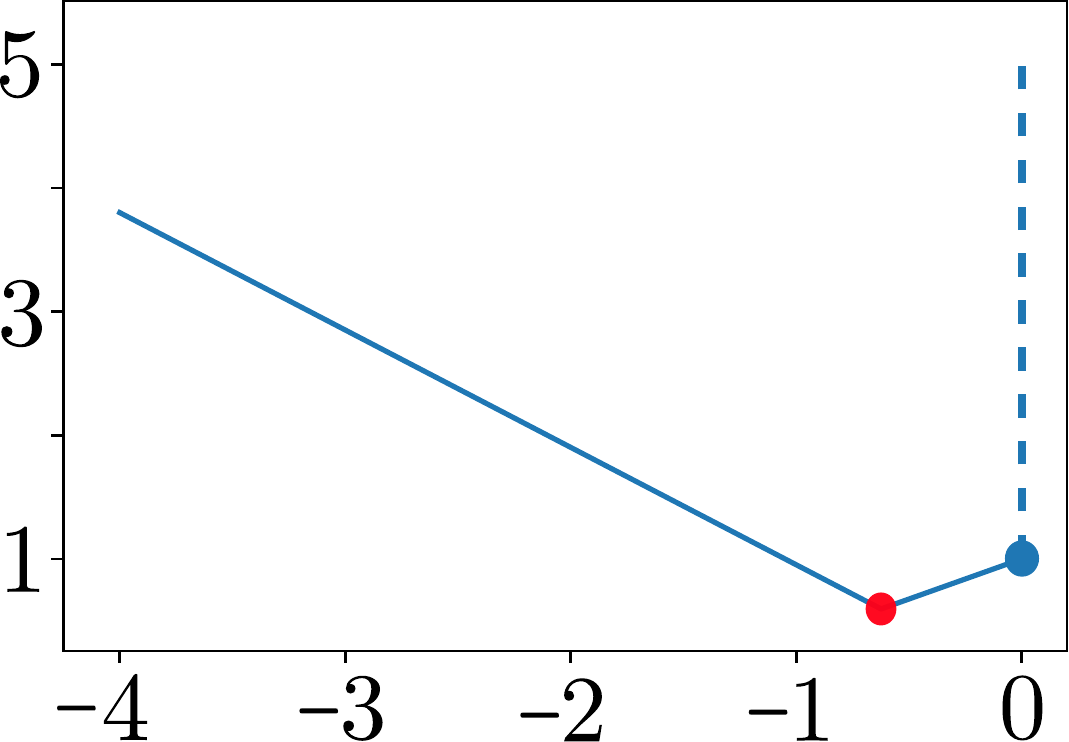}}
\caption{$F(\alpha)$ for $a=-3$ and different values of $\lambda$: (a) $\lambda = -0.2$ (b) $\lambda = 0$ (c) $\lambda = 0.5$}
\label{fig:dual}
\end{figure}

The graph of $F(\alpha)$ is slightly different depending on the sign of $\lambda$ (Fig \ref{fig:dual}), with the minimum value attained at $\alpha = 0$ if $\lambda \leq 0$ and at $\alpha = \frac{1}{e^a - e^{\lambda}}$ if $\lambda > 0$. Substituting we get the result in eq (\ref{eq:cavina}).

\[
\sigma_+ = \left\{
\begin{array}{ll}
1, & \mbox{if }\lambda \leq 0 \\
\frac{1-e^a}{e^{\lambda} - e^a}, &\mbox{if }\lambda > 0
\end{array}
\right.
\quad\quad\mbox{or}\quad
\sigma_+ = \min\left\{1, \frac{1-e^a}{e^{\lambda} - e^a}\right\}
\]

\subsection{An example with two variables}

As a final example, let us consider the problem of finding an upper bound for $\expect{e^X + e^Y}$ subjected to $\expect{X} = 0$, $\expect{XY} = \nicefrac{1}{2}$ and $X,Y\in [-1,1]$. Translating the problem into our framework we have $m=n=2$, $f(x, y) = (x, xy)$, $\phi = (0, \nicefrac{1}{2})$, $\mathcal{S} = [-1,1]^2$ and $\mathcal{D}(\mathcal{S}) = \hull{f(\mathcal{S})} = [-1,1]^2$. Applying theorem \ref{th:numerical}, this lower bound is

\[
\inf_{\alpha,\beta\in\mathds{R}} F_+(\alpha, \beta)\quad\quad\mbox{where}\quad\quad F_+(\alpha, \beta) = \sup_{x,y\in [-1,1]} G(x,y;\alpha, \beta) = \sup_{x,y\in [-1,1]} \left(e^x + e^y + \alpha x + \beta\left(xy-\frac{1}{2}\right)\right)
\]
A graph of $F_+(\alpha, \beta)$ can be found in figure \ref{fig:F-2vars}

\begin{figure}[H]
\centering
\subfigure[]{\includegraphics[height=0.22\textheight]{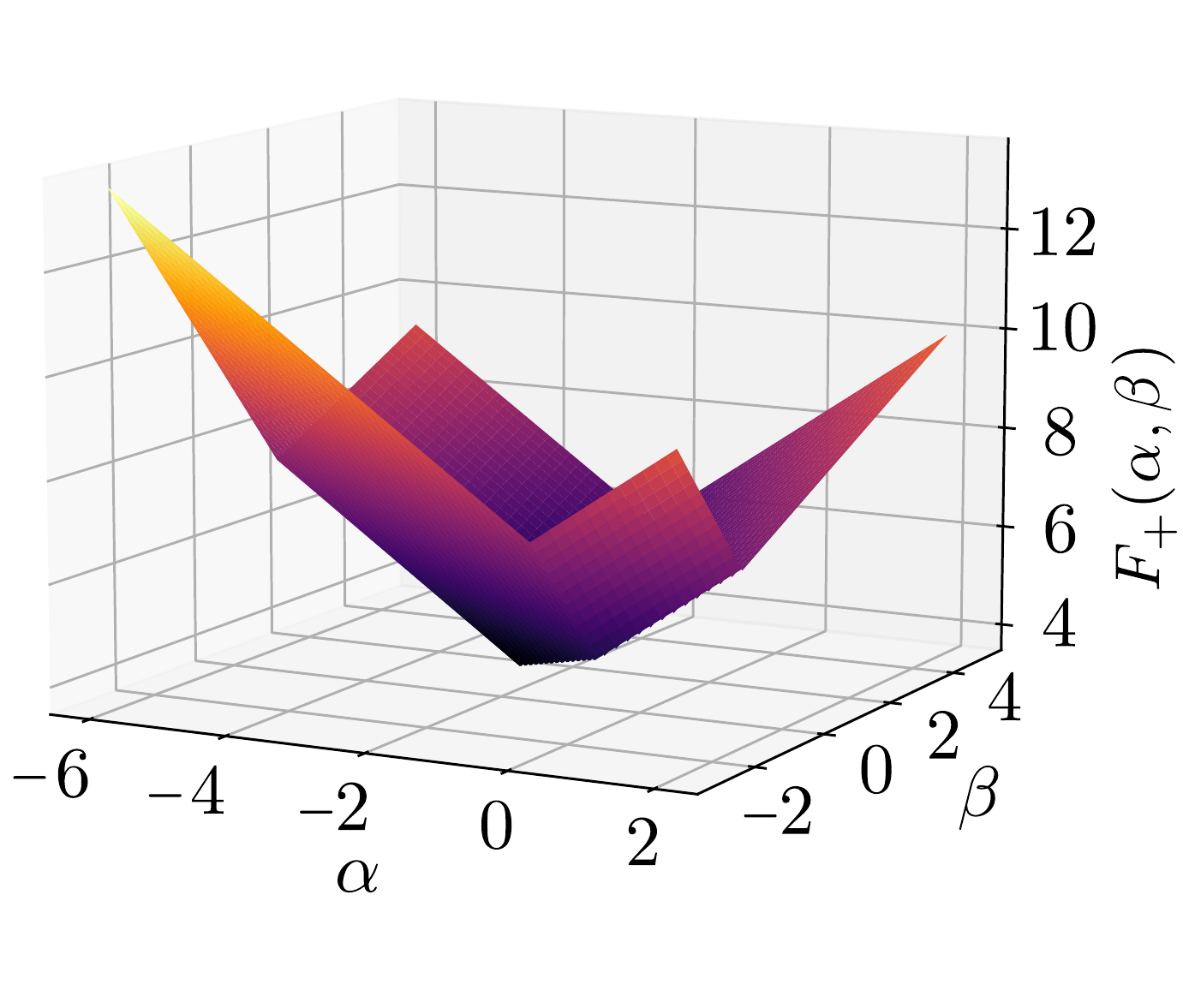}}\quad\quad
\subfigure[]{\includegraphics[height=0.22\textheight]{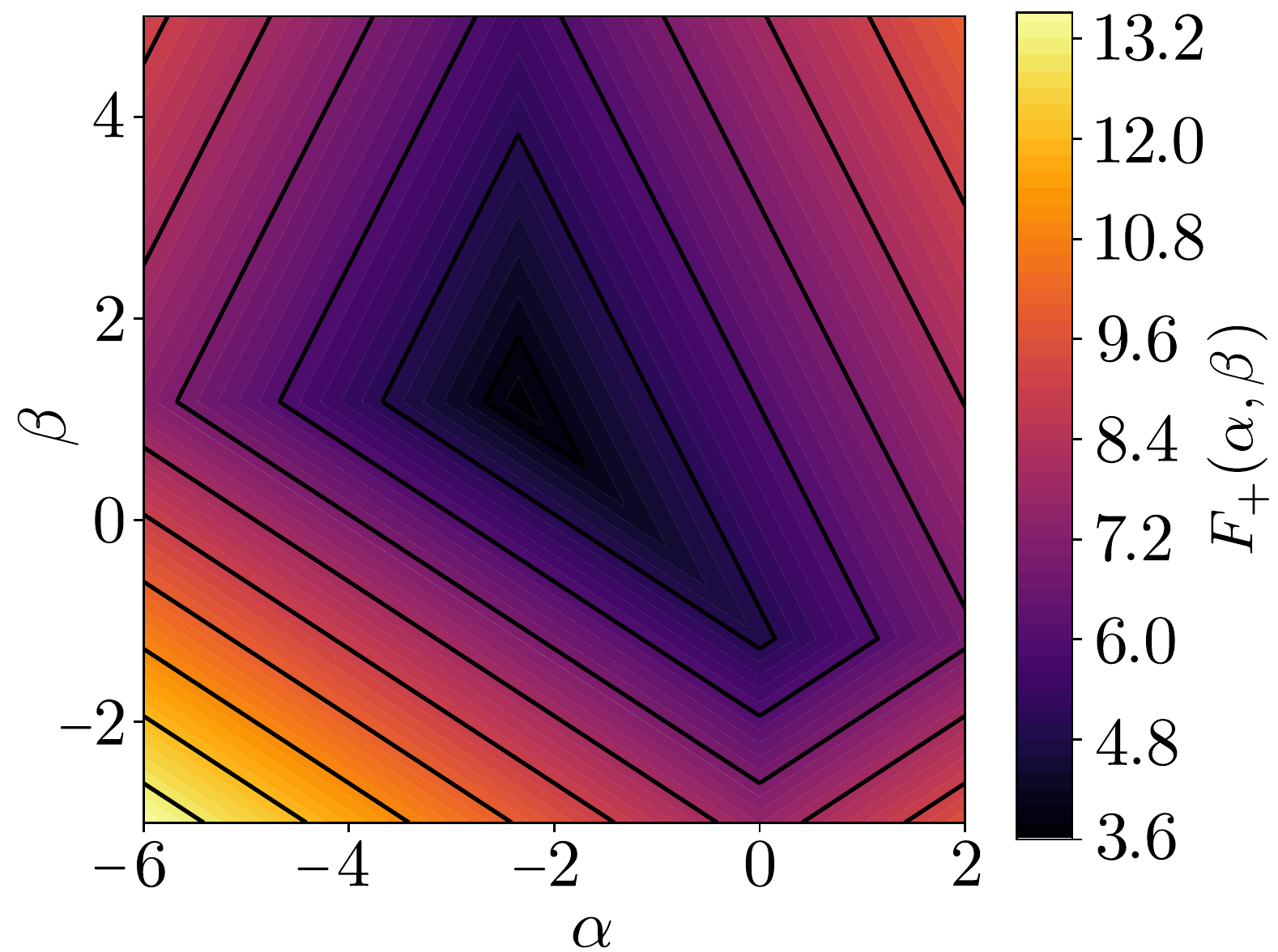}}
\caption{A graph of $F_+(\alpha, \beta)$ around $(\alpha^*, \beta^*) = \argmin F_+$. In (a) we can see that the graph consists of 4 affine regions. Each one of them corresponds to a different point $(x,y)$ that maximizes $G(x,y;\alpha, \beta) = e^x + e^y + \alpha x + \beta\left(xy-\nicefrac{1}{2}\right)$. The point $(\alpha^*, \beta^*)$ is more clearly identified in (b) that displays the contour lines of $F_+$.}
\label{fig:F-2vars}
\end{figure}

Figure \ref{fig:F-2vars}(a) is an example of what happens when there are regions where changing the vector $\alpha$ doesn't change the point $x$ that maximizes $G(x;\alpha)$ (thinking in the general case where $\alpha\in\mathds{R}^m$ and $x\in\mathds{R}^n$). Since $F(\alpha)$ becomes of the form $g(x^*) + \innp{\alpha}{(f(x^*) - \phi)}$ for some fixed $x^*$, then $F(\alpha)$ is affine in that region (the same thing happens in the example of section \ref{ssec:cavina}). The points $(x^*,y^*)$ corresponding to each region in this case are the vertexes of $\mathcal{S}$ (figure \ref{fig:2vars-regions})

\begin{figure}[hbtp]
\centering
\includegraphics[width=0.5\textwidth]{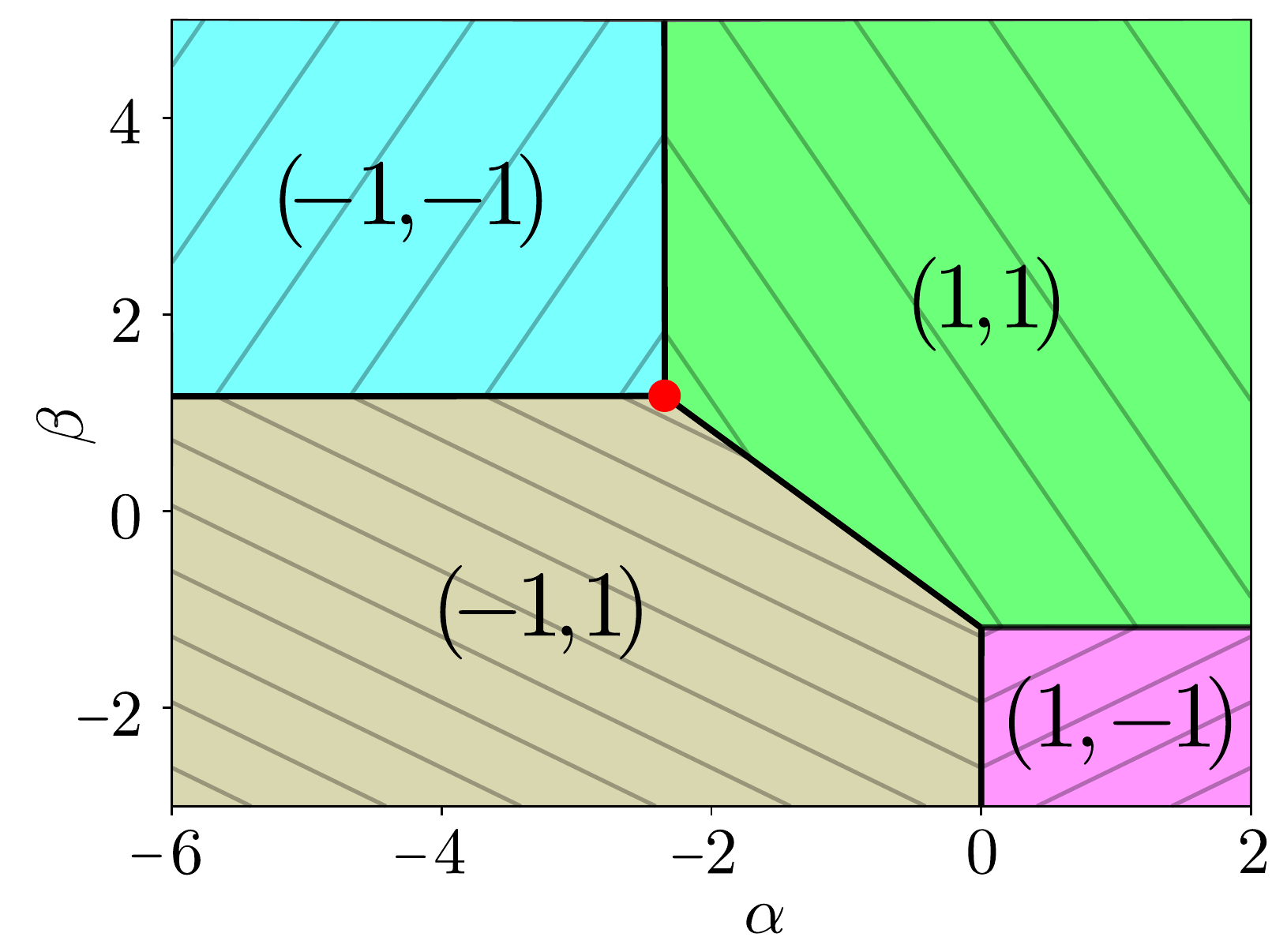}
\caption{The different points $(x^*,y^*)$ corresponding to each of the regions in figure \ref{fig:F-2vars} (each in a different color). The contour lines are included for reference, as well as the optimum $(\alpha^*, \beta^*)$ (in red)}
\label{fig:2vars-regions}
\end{figure}

In fact, knowing these points allows us to obtain $(\alpha^*, \beta^*)$, the optimal bound and even the distribution satisfying equality analytically \footnote{this is done by solving the system $G(-1, -1; \alpha^*, \beta^*) = G(1, 1; \alpha^*, \beta^*) = G(-1, 1; \alpha^*, \beta^*) = g^*$}, instead of having to rely in numerical estimates:

\[
\alpha^*=\frac{1}{e}-e,\quad\quad \beta^*=\frac{1}{2}\left(e - \frac{1}{e}\right)\quad\quad\mbox{and}\quad\quad \expect{e^X + e^Y}_{\mu} \leq \frac{5e}{4} + \frac{3}{4e}\equiv g^*\quad\forall\,\, \mu\in\MSF
\]
To identify the distribution, we connect what we did with theorem \ref{th:analytical}. The function prescribed by theorem \ref{th:analytical} in this case is actually (up to a positive multiplicative constant)

\[
g(x) + \innp{\left(\alpha^*, \beta^*\right)}{(x, xy)} + c^* = G(x, y; \alpha^*, \beta^*) - g^* \leq 0
\]
A graph of $G(x, y; \alpha^*, \beta^*) - g^*$ that makes the connection with theorem \ref{th:analytical} clearer is found in figure \ref{fig:2vars-G}

\begin{figure}[H]
\centering
\includegraphics[width=0.5\textwidth]{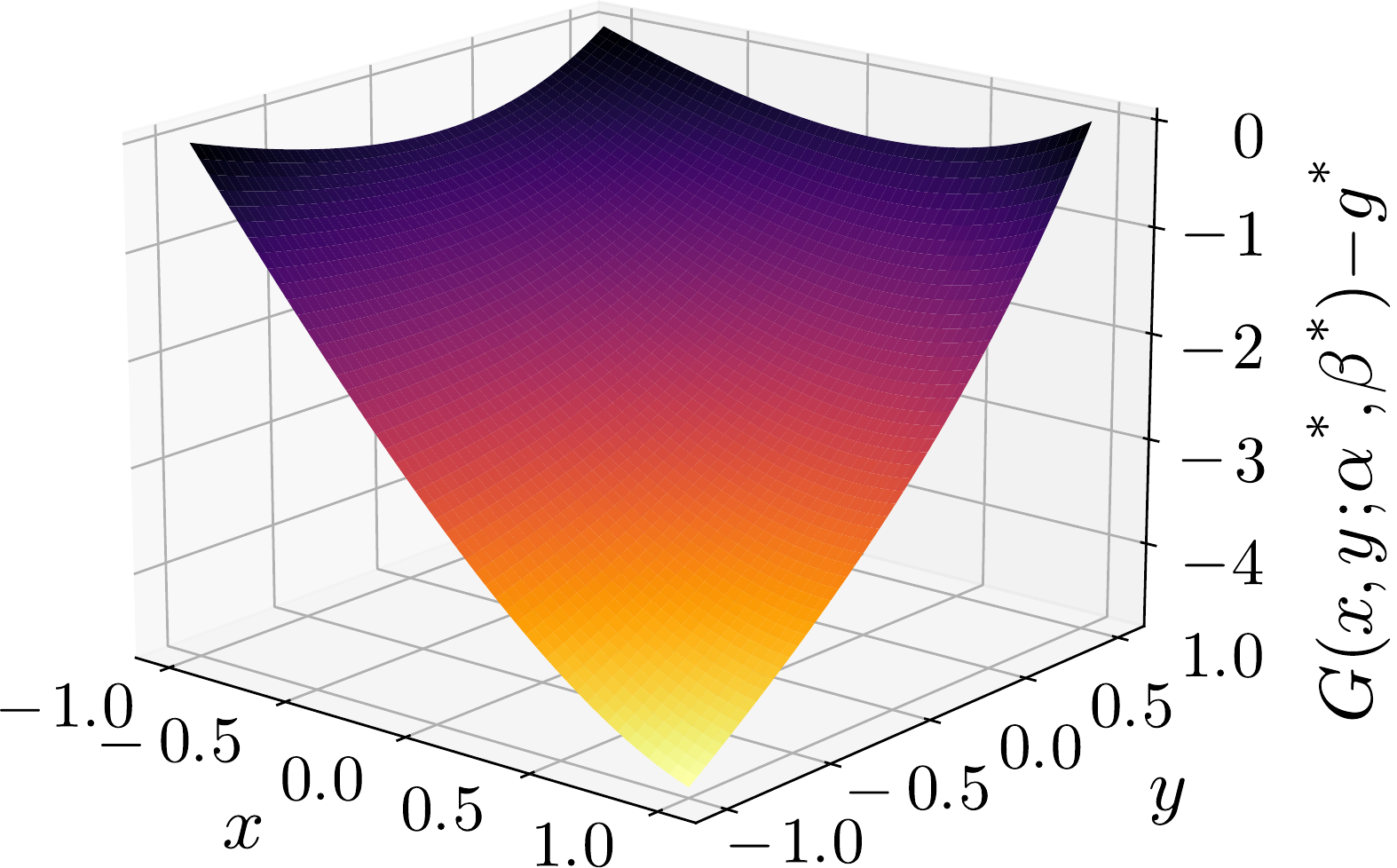}
\caption{The graph of $G(x, y; \alpha^*, \beta^*) - g^*$, showing that $G(x, y; \alpha^*, \beta^*) - g^* \leq 0\,\,\forall\,\, x\in\mathcal{S}$. We can also see that it achieves its maximum (0) exactly at the points $(-1,-1)$, $(-1,1)$ and $(1,1)$, corresponding to the regions whose intersection yields $(\alpha^*, \beta^*)$ in figure \ref{fig:2vars-regions}}
\label{fig:2vars-G}
\end{figure}

In figure \ref{fig:2vars-G} we see that the only roots of $G(x, y; \alpha^*, \beta^*) - g^*$ are $(-1,-1)$, $(-1,1)$ and $(1,1)$, corresponding to the regions whose intersection yields $(\alpha^*, \beta^*)$ in figure \ref{fig:2vars-regions}. This will be the support of the distribution that maximizes $\expect{e^X + e^Y}$ (because of theorem \ref{th:analytical}). Applying the constraints, one can easily obtain

\[
p_{(-1,-1)} = \frac{1}{4},\quad\quad p_{(-1,1)} = \frac{1}{4}\quad\quad\mbox{and}\quad\quad p_{(1,1)} = \frac{1}{2}
\]
as a distribution such that $\expect{e^X + e^Y} = g^*$.

\section{Open Questions}
\label{sec:conclusions}


As future avenues of research, we can point out the following questions that our results raise:

\begin{itemize}
\item Are there other classes of functions for which analytical results for the Jensen gap can be obtained, like in theorems \ref{th:dg-strict-convex} and \ref{th:dg-strict-convex2}? This would be interesting particularly for the cases where $n>1$.
\item Regarding some limitations of theorem \ref{th:numerical}, an interesting question is whether anything can be said in general about the case $\phi\in\bound{\mathcal{D}}$ at all.
\item The convex optimization problem that arises in theorem \ref{th:numerical} doesn't seem to have been studied in detail and even though we were able to show that it is amenable to a subgradient method, we were unable to find a way to tackle it with a higher order method (for example, it seems to be outside of the scope of barrier methods \cite{conv-opt}). As such, extending these methods to handle this new setup would be a very interesting undertaking.
\item When the $f_i$ and $g$ are polynomials, the bounds can be obtained by semidefinite programming and must coincide with the results coming from theorem \ref{th:numerical}. Whether this means that the optimization in theorem \ref{th:numerical} can be rewritten as a semidefinite program might be worth exploring.
\end{itemize}

\bibliographystyle{spmpsci}
\bibliography{jensen}

\end{document}